\documentclass[10pt]{amsart}
\usepackage{amssymb}

\newtheorem{thm}{Theorem}[section]
\newtheorem{lemma}[thm]{Lemma}
\newtheorem{proposition}[thm]{Proposition}

\newtheorem{definition}[thm]{Definition}

\newtheorem{notation}[thm]{Notation}

\newtheorem{recursionhyp}[thm]{Recursion Hypothesis}

\newcommand{\p}{\mathbb{P}}
\newcommand{\q}{\mathbb{Q}}
\newcommand{\R}{\mathbb{R}}

\newcommand{\s}{\mathbb{S}}

\newcommand{\ot}{\mathrm{ot}}
\newcommand{\cf}{\mathrm{cf}}
\newcommand{\cof}{\mathrm{cof}}
\newcommand{\dom}{\mathrm{dom}}

\newcommand{\add}{\textrm{Add}}
\newcommand{\even}{\mathrm{even}}
\newcommand{\odd}{\mathrm{odd}}

\newcommand{\col}{\mathrm{Col}}
\newcommand{\restrict}{\upharpoonright}

\begin{document}

\title{The Harrington-Shelah model with large continuum}

\author{Thomas Gilton and John Krueger}

\address{Thomas Gilton \\ Department of Mathematics \\
	University of California, Los Angeles\\
	Box 951555\\
	Los Angeles, CA 90095-1555}
\email{tdgilton@math.ucla.edu}

\address{John Krueger \\ Department of Mathematics \\ 
	University of North Texas \\
	1155 Union Circle \#311430 \\
	Denton, TX 76203}
\email{jkrueger@unt.edu}

\date{December 2017; revised December 2018}



\thanks{The second author was partially supported by 
	the National Science Foundation Grant
	No. DMS-1464859.}

\begin{abstract}
We prove from the existence of a Mahlo cardinal the 
consistency of the 
statement that $2^\omega = \omega_3$ holds and 
every stationary subset of 
$\omega_2 \cap \cof(\omega)$ reflects to an ordinal 
less than $\omega_2$ with cofinality $\omega_1$. 
\end{abstract}

\maketitle

Let us say that \emph{stationary set reflection holds at $\omega_2$} if for any stationary 
set $S \subseteq \omega_2 \cap \cof(\omega)$ 
there is an ordinal $\alpha \in \omega_2 \cap \cof(\omega_1)$ 
such that $S \cap \alpha$ is stationary in $\alpha$ (that is, $S$ \emph{reflects} to $\alpha$). 
In a classic forcing construction, Harrington and Shelah \cite{HS} proved the 
equiconsistency of stationary set reflection at $\omega_2$ with the existence 
of a Mahlo cardinal. 
Specifically, if stationary set reflection holds at $\omega_2$, then $\Box_{\omega_1}$ fails, 
and hence $\omega_2$ is a Mahlo cardinal in $L$. 
Conversely, if $\kappa$ is a Mahlo cardinal, then the generic extension obtained by 
L\'evy collapsing $\kappa$ to become $\omega_2$ and then iterating to kill the stationarity 
of nonreflecting sets satisfies stationary set reflection at $\omega_2$. 
The Harrington-Shelah argument is notable because the majority of 
stationary set reflection principles are derived by extending large cardinal 
elementary embeddings, and thus use large cardinal principles much stronger 
than the existence of a Mahlo cardinal.

The original Harrington-Shelah model satisfies the generalized continuum hypothesis, and in particular, 
that $2^\omega = \omega_1$. 
Suppose we would like to obtain a model of stationary set reflection at $\omega_2$ 
together with $2^\omega = \omega_2$. 
A natural construction would be to iterate forcing 
with countable support of length a 
weakly compact cardinal $\kappa$, alternating between adding reals and 
collapsing $\omega_2$ to have size $\omega_1$. 
Such an iteration $\p$ would be proper, $\kappa$-c.c., collapse $\kappa$ to become 
$\omega_2$, and satisfy that $2^\omega = \omega_2$. 
The fact that stationary set reflection holds in any generic extension $V[G]$ by $\p$ 
follows from the ability to 
extend an elementary embedding $j$ with critical point $\kappa$ 
after forcing with the proper forcing $j(\p) / G$ over $V[G]$.

Consider the problem of obtaining a model satisfying stationary set reflection 
at $\omega_2$ together with $2^\omega > \omega_2$. 
Since in that case not all reals would be added by the iteration collapsing $\kappa$ to 
become $\omega_2$, extending the elementary embedding becomes more difficult. 
Indeed, in the model referred to in the previous paragraph, a stronger stationary set reflection 
principle holds, namely $\textsf{WRP}(\omega_2)$, which asserts that 
any stationary subset of $[ \omega_2 ]^\omega$ reflects to $[\beta]^\omega$ 
for some uncountable $\beta < \omega_2$, and 
by a result of Todor\u{c}evi\'c, $\textsf{WRP}(\omega_2)$ implies $2^\omega \le \omega_2$ 
(see \cite[Lemma 2.9]{boban}).

In this paper we demonstrate that the cardinality of the continuum provides a natural 
separation between ordinary stationary set reflection and higher order 
reflection principles such as $\textsf{WRP}(\omega_2)$. 
We prove that, in contrast to $\textsf{WRP}(\omega_2)$, 
stationary set reflection at $\omega_2$ is consistent with $2^\omega = \omega_3$. 
This result provides a natural variation of the Harrington-Shelah model with a large value of the continuum. 
Our argument adapts the method of mixed support forcing iterations into the context 
of iterating distributive forcings. 
We expect that the technicalities worked out in this paper 
will be applicable to a broad range of similar problems.

\bigskip

We assume that the reader is familiar with the basics of forcing and has had some 
exposure to iterated forcing and proper forcing. 
Other than assuming some general knowledge of these areas, the paper is self-contained.

In Section 1 we provide an abstract definition and development of the kind of mixed 
support forcing iteration we will use in the consistency result. 
This iteration combines adding Cohen reals together with adding club subsets of $\omega_2$, 
with finite support on the Cohen forcing and supports of size 
$\omega_1$ on the club adding forcing. 
This kind of mixed support forcing iteration is reminiscent of Mitchell's classic forcing for constructing a model 
in which there is no Aronszajn tree on $\omega_2$ \cite{mitchell}, 
as well as the term forcing 
analysis provided in Abraham's extension of Mitchell's result to two successive cardinals \cite{abraham}. 

The main challenge in proving our consistency result will be to verify that the forcing 
iteration preserves $\omega_1$ and $\omega_2$. 
In Section 2 we analyze the features of this kind of forcing iteration relevant 
to the issue of cardinal preservation. 
In Section 3 we put the pieces worked out in Sections 1 and 2 together to prove the consistency of stationary set 
reflection at $\omega_2$ together with $2^\omega = \omega_3$.

\section{Suitable Mixed Support Forcing Iterations}

In this section we introduce and develop the basic properties of 
the type of mixed support forcing iteration which we will use 
in the consistency result. 
This kind of iteration will alternate between adding Cohen subsets of 
$\omega$ and adding clubs disjoint from certain subsets of $\omega_2$. 
The support of a condition in such an iteration will be finite on the 
Cohen part and of size less than $\omega_2$ on the club adding part.

We let \emph{even} denote the class of even ordinals, 
and \emph{odd} the class of odd ordinals.

\begin{definition}
Let $\alpha \le \omega_3$. 
Let $\langle \p_\beta : \beta \le \alpha \rangle$ be a sequence of forcing posets and 
$\langle \dot S_\gamma : \gamma \in \alpha \cap \odd \rangle$ a sequence such that 
for all odd $\gamma < \alpha$, $\dot S_\gamma$ is a nice $\p_\gamma$-name for a subset 
of $\omega_2 \cap \cof(\omega)$. 
Assume that for all $\beta \le \alpha$, every member of $\p_\beta$ is a function 
whose domain is a subset of $\beta$, and define 
$$
\p_\beta^c := \{ p \in \p_\beta : \dom(p) \subseteq \even \}.
$$
We say that the sequence of forcing posets is a 
\emph{suitable mixed support forcing iteration of length $\alpha$ based on the sequence of names} 
if the following statements are satisfied:
\begin{enumerate}
	\item $\p_0 = \{ \emptyset \}$ is the trivial forcing;
	\item if $\gamma < \alpha$ is even, then $p \in \p_{\gamma+1}$ iff $p$ is a function 
	whose domain is a subset of $\gamma+1$ such that $p \restrict \gamma \in \p_\gamma$ and, 
	if $\gamma \in \dom(p)$, then $p(\gamma) \in \add(\omega)$;
	\item if $\gamma < \alpha$ is odd, then $p \in \p_{\gamma+1}$ iff 
	$p$ is a function 
	whose domain is a subset of $\gamma+1$ such that
	$p \restrict \gamma \in \p_\gamma$ and, 
	if $\gamma \in \dom(p)$, then $p(\gamma)$ is a nice $\p_\gamma^c$-name for a 
	nonempty closed and bounded subset of $\omega_2$ such that 
	$$
	p \restrict \gamma \Vdash_{\p_\gamma} 
	p(\gamma) \cap \dot S_\gamma = \emptyset;
	$$
	\item if $\delta \le \alpha$ is a limit ordinal, then $p \in \p_\delta$ iff 
	$p$ is a function 
	whose domain is a subset of $\delta$ such that
	$|\dom(p) \cap \even| < \omega$, $|\dom(p) \cap \odd| < \omega_2$, and for all $\beta < \delta$, 
	$p \restrict \beta \in \p_\beta$;
	\item for all $\beta \le \alpha$, $q \le p$ in $\p_\beta$ iff 
	$\dom(p) \subseteq \dom(q)$, 
	and for all $\gamma \in \dom(p)$, if $\gamma$ is even 
	then $p(\gamma) \subseteq q(\gamma)$, 
	and if $\gamma$ is odd then 
	$$
	q \restrict (\gamma \cap \even) \Vdash_{\p_\gamma^c} q(\gamma) \ \textrm{is an end-extension of} \ p(\gamma).
	$$
	\end{enumerate}
\end{definition}

The definition makes sense without assuming that the forcing iterations preserve cardinals, 
if we interpret $\omega_2$ in the definition as meaning  $\omega_2$ of the ground model. 
In any case, the only such forcing iterations we will consider in this paper will preserve $\omega_1$ and 
$\omega_2$, although cardinal preservation 
will not be verified until the end of the paper.

The requirement in (3) that $p(\gamma)$ is a nice $\p_\gamma^c$-name, rather than  
a $\p_\gamma$-name, is made in order 
to prove the following absoluteness result.

\begin{lemma}
	Let $M$ be a transitive model of $\textsf{ZFC} - \textsf{Powerset}$ with $\omega_2 \in M$ and 
	$M^{\omega_1} \subseteq M$. 
	Suppose that 
	$\langle \p_\beta : \beta \le \alpha \rangle$ 
	is a sequence of forcing posets in $M$ and 
	$\langle \dot S_\gamma : \gamma \in \alpha \cap \odd 
	\rangle$ is a sequence in $M$ so that for each odd 
	$\gamma \in \alpha$, $\dot S_\gamma$ is a nice 
	$\p_\gamma$-name for a subset of 
	$\omega_2 \cap \cof(\omega)$. 
	Then $\langle \p_\beta : \beta \le \alpha \rangle$ 
	is a suitable mixed support forcing iteration 
	based on the sequence of names 
	$\langle \dot S_\gamma : \gamma \in \alpha \cap \odd 
	\rangle$ iff 
	$M$ models that it is.
\end{lemma}

The proof, which we omit, is a straightforward verification 
that each property of Definition 1.1 is absolute between $M$ and $V$. 
The closure of $M$ is 
used to see that $M$ contains all names described in Definition 1.1(3) (see Lemma 1.3 below).

\bigskip

For the remainder of the section we fix a particular suitable mixed support forcing iteration 
$\langle \p_\beta : \beta \le \alpha \rangle$ based on a sequence of names 
$\langle \dot S_\gamma : \gamma \in \alpha \cap \odd \rangle$. 
For $\beta \le \alpha$, 
we will write $q \le_\beta p$ to mean that $q \le p$ in $\p_\beta$, and we will 
abbreviate $\Vdash_{\p_\beta}$ as $\Vdash_\beta$.

When $p$ is a condition in $\p_\beta$ and $\gamma < \beta$, for simplicity 
we will sometimes write 
$p(\gamma)$ without knowing whether or not $\gamma \in \dom(p)$; 
in the case that it is not, then $p(\gamma)$ means the empty set.

The proof of the next lemma is straightforward.

\begin{lemma}
	Let $\beta \le \alpha$. 
	The forcing poset $\p_\beta^c$ is a regular suborder of $\p_\beta$, and $\p_\beta^c$ is 
	isomorphic to $\add(\omega,\ot(\beta \cap \even))$.
\end{lemma}

It follows that if $G$ is a generic filter on $\p_\beta$, then $G^c := G \cap \p_\beta^c$ 
is a generic filter on $\p_\beta^c$. 
Also, for any condition $q \in G$, $q \le_\beta (q \restrict \even)$ implies 
that $q \restrict \even \in G^c$. 
If $\dot x$ is a $\p_\beta^c$-name, then it is also a $\p_\beta$-name and $\dot x^G = \dot x^{G^c}$.

The next two lemmas state some basic facts about the forcing iteration. 
The proofs, which we omit, are straightforward.

\begin{lemma}
Let $\gamma < \beta \le \alpha$.
\begin{enumerate}
	\item $\p_\gamma \subseteq \p_\beta$, and for all $p \in \p_\beta$, 
	$p \restrict \gamma \in \p_\gamma$;
	\item if $p$ and $q$ are in $\p_\gamma$, then 
	$q \le_\gamma p$ iff $q \le_\beta p$;
	\item if $p \in \p_\gamma$, $r \in \p_\beta$, and $r \le_\beta p$, then 
	$r \restrict \gamma \le_\gamma p$;
	\item if $q \in \p_\beta$ and $r \le_\gamma q \restrict \gamma$, then 
	$r \cup q \restrict [\gamma,\beta)$ is in $\p_\beta$ and is $\le_\beta$-below 
	$r$ and $q$;
	\item $\p_\gamma$ is a regular suborder of $\p_\beta$.
	\end{enumerate}	
	\end{lemma}

\begin{lemma}
	Let $\beta \le \alpha$ and 
	$p$ and $q$ be in $\p_\beta$.
	\begin{enumerate}
		\item If $\beta$ is a limit ordinal, then $q \le_\beta p$ iff for all 
		$\gamma < \beta$, $q \restrict \gamma \le_\gamma p \restrict \gamma$;
		\item if $\beta = \gamma+1$, where $\gamma$ is even, then 
		$q \le_\beta p$ iff $q \restrict \gamma \le_\gamma p \restrict \gamma$ and 
		$p(\gamma) \subseteq q(\gamma)$;
		\item if $\beta = \gamma+1$, where $\gamma$ is odd, then 
		$q \le_\beta p$ iff $q \restrict \gamma \le_\gamma p \restrict \gamma$ 
		and $q \restrict (\gamma \cap \even)$ forces in $\p_\gamma^c$ that 
		$q(\gamma)$ is an end-extension of $p(\gamma)$.
		\end{enumerate}
	\end{lemma}

\begin{notation}
	Let $\beta \le \alpha$. 
	For $p$ and $q$ in $\p_\beta$, let $q \le_\beta^* p$ mean that $q \le_\beta p$ 
	and $q \restrict \even = p \restrict \even$. 
	For $p$ and $q$ in $\p_\beta^c$, let $q \le_\beta^c p$ mean that 
	$q \le_\beta p$. 
	We will abbreviate the forcing poset $(\p_\beta,\le_\beta^*)$ as $\p_\beta^*$ and 
	$(\p_\beta^c,\le_\beta^c)$ as $\p_\beta^c$.
\end{notation}

Consider $p \in \p_\beta$ and $a \in \p_\beta^c$. 
Then $a$ and $p$ are compatible in $\p_\beta$ iff 
$a$ and $p \restrict \even$ are compatible in $\p_\beta^c$ iff 
for all even $\gamma \in \dom(p) \cap \dom(a)$, $p(\gamma)$ and $a(\gamma)$ are 
compatible in $\add(\omega)$, that is, $p(\gamma) \cup a(\gamma)$ 
is a function.

\begin{notation}
	Let $\beta \le \alpha$. 
	If $a \in \p_\beta^c$ and $p \in \p_\beta$, and $a$ and $p$ are compatible in $\p_\beta$, 
	let $p + a$ denote the function $s$ such that $\dom(s) := \dom(a) \cup \dom(p)$, 
	for all even $\gamma \in \dom(s)$, $s(\gamma) := a(\gamma) \cup p(\gamma)$, and for all 
	odd $\gamma \in \dom(s)$, $s(\gamma) := p(\gamma)$.
\end{notation}

The proofs of the next four lemmas are straightforward.

\begin{lemma}
	Let $\beta \le \alpha$. 
	If $a \in \p_\beta^c$ and $p \in \p_\beta$, and $a$ and $p$ are compatible in $\p_\beta$, then 
	$p + a$ is in $\p_\beta$ and $p + a \le_\beta p, a$. 
	Moreover, $p + a$ is the greatest lower bound of $p$ and $a$.
	\end{lemma}

\begin{lemma}
	Let $\beta \le \alpha$. 
	Let $p \in \p_\beta$ and $a \in \p_\beta^c$. 
	Let $G$ be a generic filter on $\p_\beta$. 
	If $p$ and $a$ are both in $G$, then so is $p + a$.
\end{lemma}

\begin{lemma}
	Let $\beta \le \alpha$. 
	\begin{enumerate}
		\item For all $p \in \p_\beta$, $p \le_\beta p \restrict \even$;
		\item if $q \le_\beta p$ then $q \restrict \even \le_\beta^c p \restrict \even$;
		\item if $q \le_\beta^* p$, $a \in \p_\beta^c$, and $a$ and $p$ are compatible in $\p_\beta$, 
		then $a$ and $q$ are compatible in $\p_\beta$ and $q + a \le_\beta p + a$.
		\end{enumerate}
	\end{lemma}

\begin{lemma}
	Let $\beta \le \alpha$. 
	Suppose that $b \le_\beta^c a$ and $q \le_\beta p$, 
	where $a$ and $p$ are compatible in $\p_\beta$ 
	and $b$ and $q$ are compatible in $\p_\beta$. 
	Then $q + b \le_\beta p + a$.
\end{lemma}

\begin{lemma}
	Let $\beta \le \alpha$, $q \in \p_\beta$, $\dot x$ a $\p_\beta^c$-name, and $\varphi(x)$ a 
	$\Delta_0$-formula. Then 
	$$
	q \Vdash_\beta \varphi(\dot x) \ \ \textrm{iff} \ \ (q \restrict \even) \Vdash_{\p_\beta^c} 
	\varphi(\dot x).
	$$
\end{lemma}

\begin{proof}
	For the backwards implication, 
	assume that 
	$q \restrict \even$ forces in $\p_\beta^c$ 
	that $\varphi(\dot x)$ holds. 
	If $G$ is a generic filter on $\p_\beta$ which contains 
	$q$, then 
	$q \restrict \even \in G^c$ implies that $\dot x^{G^c} = \dot x^G$ satisfies $\varphi$ in 
	$V[G^c]$ and hence in $V[G]$. 
	For the forward implication, suppose that $q$ forces in $\p_\beta$ that $\varphi(\dot x)$ holds. 
	Consider any $b \le_\beta^c q \restrict \even$. 
	Fix a generic filter $G$ on 
	$\p_\beta$ which contains $q + b$, and let 
	$x := \dot x^G = \dot x^{G^c}$. 
	Since $q + b \le_\beta q$, $q \in G$, and therefore 
	$\varphi(x)$ holds in $V[G]$ and hence in $V[G^c]$. 
	But $q + b \le_\beta b$ implies that $b \in G \cap \p_\beta^c = G^c$. 
	Thus, $b$ does not force the negation of 
	$\varphi(\dot x)$. 
	Since $b$ was arbitrary, $q \restrict \even$ forces in $\p_\beta^c$ that 
	$\varphi(\dot x)$ holds.
\end{proof}

In particular, in Definition 1.1(5) the property 
$$
q \restrict (\gamma \cap \even) \Vdash_{\p_\gamma^c} q(\gamma) \ \textrm{is an end-extension of} \ p(\gamma)
$$
is equivalent to 
$$
q \restrict \gamma \Vdash_\gamma q(\gamma) \ \textrm{is an end-extension of} \ p(\gamma).
$$

The next technical proposition will be crucial to the arguments in Section 2.

\begin{proposition}
	Let $\beta \le \alpha$. 
	Suppose that $q \le_\beta p$. 
	Let $b := q \restrict \even$. 
	Then there exists $q' \in \p_\beta$ such that  
	$$
	q \le_\beta q' \le_\beta^* p
	$$
	and 
	$$
	q \le_\beta q' + b \le_\beta q.
	$$
\end{proposition}

\begin{proof}
	Let $q' \restrict \even := p \restrict \even$. 
	Let $\dom(q') \cap \odd := \dom(q) \cap \odd$. 
	Consider $\gamma \in \dom(q') \cap \odd$. 
	By the maximality principle for names, 
	we can find a nice $\p_\gamma^c$-name $q'(\gamma)$ for a nonempty closed and bounded subset of 
	$\omega_2$ which end-extends $p(\gamma)$ such that, 
	if $b \restrict \gamma$ 
	is in the generic filter on $\p_\gamma^c$, then $q'(\gamma) = q(\gamma)$, and otherwise 
	$q'(\gamma)$ is $p(\gamma)$ together with the least ordinal of cofinality $\omega_1$ 
	strictly above all members of $p(\gamma)$.
	
	Assume for a moment that $q'$ is a condition. 
	Note that for all odd $\gamma \in \dom(q')$, 
	$q \restrict (\gamma \cap \even) = b \restrict \gamma$ 
	forces that $q'(\gamma) = q(\gamma)$. 
	Based on this fact, it is easy to check that 
	$q \le_\beta q'$. 
	Also, $q' \restrict \even = p \restrict \even$, 
	and for all odd $\gamma \in \dom(q')$, 
	$\p_\gamma^c$ forces that 
	$q'(\gamma)$ is an end-extension of $p(\gamma)$. 
	It easily follows that $q' \le_\beta^* p$, which 
	verifies the first pair of inequalities.
		
	For the second pair, 
	since $q \le_\beta p$, 
	$b = q \restrict \even \le_\beta^c p \restrict \even = q' \restrict \even$. 
	So $b$ and $q'$ are compatible in $\p_\beta$. 
	Also, $q \le_\beta q'$ from the previous paragraph. 
	By Lemma 1.11, $q = q + b \le_\beta q' + b$. 
	Now if $\gamma \in \dom(q')$ is odd, and assuming 
	$(q' + b) \restrict \gamma \le_\gamma 
	q \restrict \gamma$, it follows that 
	$(q' + b) \restrict (\gamma \cap \even) = b \restrict 
	\gamma$ forces that $q'(\gamma) = q(\gamma)$, and 
	hence $(q' + b) \restrict (\gamma+1) \le_{\gamma+1} 
	q \restrict (\gamma+1)$. 
	It easily follows by an inductive argument that $q' + b \le_\beta q$.
	
	Thus, we have shown that if $q' \in \p_\beta$, then 
	all of the inequalities stated in the proposition hold. 
	Moreover, the above argument also shows that if, 
	for a fixed 
	$\xi \le \beta$, $q' \restrict \xi \in \p_\xi$, 
	then all of the inequalities 
	stated in the proposition hold for the 
	conditions restricted to $\xi$.
	
	It remains to show that $q'$ is a condition. 
	By Definition 1.1, it suffices to show that 
	whenever $\gamma \in \dom(q')$ is odd, if we assume 
	that $q' \restrict \gamma$ is in $\p_\gamma$ and is $\le_\gamma^*$-below 
	$p \restrict \gamma$, then 
	$$
	q' \restrict \gamma \Vdash_\gamma 
	q'(\gamma) \cap \dot S_\gamma =\emptyset.
	$$
	Let $G$ be a generic filter on $\p_\gamma$ which contains 
	$q' \restrict \gamma$. 
	Let $S_\gamma := \dot S_\gamma^G$, 
	$G^c := G \cap \p_\gamma^c$, 
	and $x := q'(\gamma)^{G^c}$. 
	We will show that $x \cap S_\gamma = \emptyset$.
	
	By the choice of $q'(\gamma)$, $x$ is equal to 
	$q(\gamma)^{G^c}$ provided that 
	$b \restrict \gamma \in G^c$, and otherwise 
	is equal to $p(\gamma)^{G^c}$ together with an ordinal of cofinality $\omega_1$. 
	In the latter case, since 
	$q' \restrict \gamma \le_\gamma 
	p \restrict \gamma$ and 
	$p \restrict \gamma \Vdash_\gamma 
	p(\gamma) \cap \dot S_\gamma = \emptyset$, we have 
	that $p \restrict \gamma \in G$ and 
	$p(\gamma)^{G^c}$ is disjoint from 
	$S_\gamma$. 
	Since $x$ is equal to $p(\gamma)^{G^c}$ together with an ordinal 
	of cofinality $\omega_1$, whereas $S_\gamma$ consists of ordinals of cofinality $\omega$, 
	$x$ is disjoint from $S_\gamma$. 
	So assume that $b \restrict \gamma \in G^{c}$. 
	Then by Lemma 1.9, 
	$(q' \restrict \gamma) + (b \restrict \gamma) \in G$. 
	But this condition is $\le_\gamma$-below 
	$q \restrict \gamma$. 
	So $q \restrict \gamma \in G$. 
	As $q \restrict \gamma$ forces in $\p_\gamma$ that 
	$q(\gamma) \cap \dot S_\gamma = \emptyset$, it follows 
	that $q(\gamma)^{G^c} = q'(\gamma)^{G^c} = x$ is disjoint 
	from $S_\gamma$.
	\end{proof}

\begin{definition}
	Let $\beta \le \alpha$. 
	Define $\p_\beta^c \otimes \p_\beta^*$ as the forcing poset consisting of pairs 
	$(a,p)$, where $a \in \p_\beta^c$ and $p \in \p_\beta$, such that 
	$a$ and $p$ are compatible in $\p_\beta$, 
	with the ordering $(a_1,p_1) \le (a_0,p_0)$ if
	$a_1 \le_\beta^c a_0$ and $p_1 \le_\beta^* p_0$.
	\end{definition}

Observe that if $p \in \p_\beta$, then 
$(p \restrict \even,p) \in \p_\beta^c \otimes \p_\beta^*$.

For any forcing poset $\q$ and $q \in \q$, 
we will use the notation $\q / q$ for the suborder 
$\{ r \in \q : r \le_\q q \}$.

The next lemma reveals that $\p_\beta^c \otimes \p_\beta^*$ is 
essentially a product forcing.

\begin{lemma}
	Let $\beta \le \alpha$. 
	Let $(a,p) \in \p_\beta^c \otimes \p_\beta^*$, 
	and assume that $a \le_\beta^c p \restrict \even$. 
	Then $(\p_\beta^c \otimes \p_\beta^*) / (a,p)$ 
	is equal to the product forcing 
	$$
	(\p_\beta^c / a) \times (\p_\beta^* / p).
	$$
\end{lemma}

\begin{proof}
	Let $(b,q) \le (a,p)$ in $\p_\beta^c \otimes \p_\beta^*$. 
	Then $b \le_\beta^c a$ and $q \le_\beta^* p$. 
	Thus, 
	$(b,q) \in (\p_\beta^c / a) \times (\p_\beta^* / p)$.
	
	Now consider $(b,q) \in 
	(\p_\beta^c / a) \times (\p_\beta^* / p)$. 
	Then $b \le_\beta^c a$ and $q \le_\beta^* p$. 
	By the choice of $(a,p)$, 
	$b \le_\beta^c a \le_\beta^c p \restrict \even = 
	q \restrict \even$, and in particular, 
	$b$ and $q$ are compatible in $\p_\beta$. 
	Therefore, $(b,q)$ is in $\p_\beta^c \otimes \p_\beta^*$. 
	And $b \le_\beta^c a$ and $q \le_\beta^* p$ means that 
	$(b,q) \le (a,p)$ in $\p_\beta^c \otimes \p_\beta^*$.
	Finally, it is immediate by definition that these 
	two forcings have the same ordering.
	\end{proof}

Note that there are densely many conditions $(a,p)$ 
in $\p_\beta^c \otimes \p_\beta^*$ such that 
$a \le_\beta^c p \restrict \even$. 
This observation together with Lemma 1.15 easily 
implies the next result.

\begin{lemma}
	Let $\beta \le \alpha$. 
	Suppose that $H$ is a generic filter on 
	$\p_\beta^c \otimes \p_\beta^*$. 
	Then there is a condition $(a,p) \in H$ such that 
	$a \le_\beta^c p \restrict \even$. 
	Moreover, if $(a,p)$ is any such condition in $H$, 
	then letting $K := H \cap 
	((\p_\beta^c \otimes \p_\beta^*) / (a,p))$, 
	we have that $K$ is a generic filter on 
	$(\p_\beta^c / a) \times (\p_\beta^* / p)$ 
	and $V[H] = V[K]$.
\end{lemma}

To provide some additional clarification, let us describe 
the forcing poset $\p_\beta^c \otimes \p_\beta^*$ as a 
disjoint sum of product forcings. 
Namely, for each $b \in \p_\beta^c$, observe that 
$\p_\beta^* / b 
= \{ p \in \p_\beta : p \restrict \even = b \}$. 
In particular, if $b \ne c$ then 
$\p_\beta^* / b$ and $\p_\beta^* / c$ are disjoint, and 
moreover, any condition in $\p_\beta^* / b$ and any 
condition in $\p_\beta^* / c$ are $\le_\beta^*$-incomparable. 

Let $D$ be the dense set of conditions $(a,p)$ in 
$\p_\beta^c \otimes \p_\beta^*$ such that 
$a \le_\beta^c p \restrict \even$. 
It is easy to check that 
$$
D = 
\bigcup 
\{ (\p_\beta^c / b) \times (\p_\beta^* / b) : b \in \p_\beta^c \}.
$$
Thus, $\p_\beta^c \otimes \p_\beta^*$ contains a dense subset which is a disjoint sum of product forcings.

\begin{definition}
	Let $\beta \le \alpha$. 
	Define $\tau_\beta : \p_\beta^c \otimes \p_\beta^* \to \p_\beta$ by 
	$\tau_\beta(a,p) := p + a$.
	\end{definition}

Note that this definition makes sense by Lemma 1.8.

\begin{lemma}
	Let $\beta \le \alpha$. 
	The function $\tau_\beta : \p_\beta^c \otimes \p_\beta^* \to \p_\beta$ 
	is a surjective projection mapping.
	\end{lemma}

\begin{proof}
	Suppose that $(b,q) \le (a,p)$ in $\p_\beta^c \otimes \p_\beta^*$. 
	Then by definition, $b \le_\beta^c a$ and $q \le_\beta^* p$. 
	Hence, $q \le_\beta p$. 
	By Lemma 1.11, 
	$\tau_\beta(b,q) = q + b \le_\beta p + a = 
	\tau_\beta(a,p)$.
	
	Consider a condition $p \in \p_\beta$. 
	Then $(p \restrict \even,p) \in 
	\p_\beta^c \otimes \p_\beta^*$, and 
	$\tau_\beta(p \restrict \even,p) = p$. 
	So $\tau_\beta$ is surjective.
		
	Now assume that $q \le_\beta \tau_\beta(a,p) = p + a$. 
	We will find $(b,q') \le (a,p)$ in 
	$\p_\beta^c \otimes \p_\beta^*$ such that 
	$\tau_\beta(b,q') \le_\beta q$. 
	Now $q \le_\beta p + a \le_\beta p$, so $q \le_\beta p$. 
	Let $b := q \restrict \even$. 
	Then by Lemma 1.10(2), 
	$$
	b = q \restrict \even \le_\beta^c 
	(p + a) \restrict \even 
	\le_\beta^c a, p \restrict \even.
	$$
	So $b \le_\beta^c a$ and 
	$b \le_\beta^c p \restrict \even$. 
	Apply Proposition 1.13 to find $q' \in \p_\beta$ 
	such that $q \le_\beta q' \le_\beta^* p$ 
	and $q \le_\beta q' + b \le_\beta q$.
	
	Since $b \le_\beta^c p \restrict \even = 
	q' \restrict \even$, 
	$b$ and $q' \restrict \even$ 
	are compatible in $\p_\beta^c$. 
	Hence, $b$ and $q'$ are compatible in $\p_\beta$. 
	Therefore, $(b,q') \in \p_\beta^c \otimes \p_\beta^*$. 
	Also, as noted above, $b \le_\beta^c a$ and 
	$q' \le_\beta^* p$, and therefore 
	$(b,q') \le (a,p)$ in $\p_\beta^c \otimes \p_\beta^*$. 
	Finally, $\tau_\beta(b,q') = q' + b \le_\beta q$.
\end{proof}

The final result from this section will be used 
in the cardinal preservation arguments needed for the 
consistency result.

\begin{lemma}
	Assume that $2^{\omega_1} = \omega_2$. Then:
	\begin{enumerate}
		\item for all $\beta \le \alpha$ with 
		$|\beta| \le \omega_2$, $|\p_\beta| \le \omega_2$;
		\item if $\alpha = \omega_3$, then 
		$\p_\alpha = \bigcup \{ \p_\beta : 
		\beta < \omega_3 \}$ has size $\omega_3$ 
		and $\p_\alpha$ is $\omega_3$-c.c.;
		\item if $\alpha = \omega_3$, then 
		for all $a \in \p_\alpha^c$, 
		$\p_\alpha^* / a = 
		\bigcup \{ \p_\beta^* / a : \beta < \omega_3 \}$ 
		has size $\omega_3$ and is $\omega_3$-c.c.
		\end{enumerate}
	\end{lemma}

\begin{proof}
	(1) Since $\alpha \le \omega_3$, 
	for all $\gamma \in \alpha$, 
	$\p_\gamma^c$ is $\omega_1$-c.c.\! and has size at most $\omega_2$. 
	Hence, there are at most $2^{\omega_1} = \omega_2$ many nice $\p_\gamma^c$-names 
	for bounded subsets of $\omega_2$. 
	With this observation, 
	(1) easily follows by induction on $\beta$.
	
	(2) The first part of (2) easily follows from Definition 1.1. 
	If $\{ p_i : i < \omega_3 \} \subseteq \p_\alpha$, then a 
	$\Delta$-system argument implies 
	that there is a set $X \subseteq \omega_3$ of size $\omega_3$ 
	and a function $r$ such that for all $i < j$ in $X$, 
	$\dom(p_i) \cap \dom(p_j) = \dom(r)$ and for all $\gamma \in \dom(r)$, 
	$p_i(\gamma) = p_j(\gamma)$. 
	It easily follows that $p_i \cup p_j$ is a condition in $\p_\alpha$ below $p_i$ and $p_j$, 
	proving that $\p_\alpha$ is $\omega_3$-c.c.
	
	(3) The proof of (3) is similar to the proof of (2).
	\end{proof}

Note that if $\alpha = \omega_3$, then 
$\p_{\alpha}^*$ is not $\omega_3$-c.c., since any two conditions in $\p_{\alpha}^*$ 
with different even parts are incompatible in $\p_\alpha^*$.

\section{Distributivity and cardinal preservation}

The most challenging part of 
our main consistency result will be in the 
verification that a 
particular suitable mixed support forcing iteration 
$\langle \p_\beta : \beta \le \omega_3 \rangle$, which destroys the stationarity of 
nonreflecting subsets of $\omega_2 \cap \cof(\omega)$, 
preserves $\omega_1$ and $\omega_2$. 
By Propositions 2.1 and 2.2 below, it will suffice to prove that $\p_\beta^*$ is 
$\omega_2$-distributive for all $\beta < \omega_3$.

For some perspective, let us review in rough outline 
the original Harrington-Shelah argument \cite{HS}. 
Start with a model of \textsf{GCH} in which 
$\kappa$ is a Mahlo cardinal, 
and let $G$ be a generic filter on 
the L\'evy collapse $\col(\omega_1,<\! \kappa)$. 
In $V[G]$, define a forcing 
iteration $\langle \p_\alpha, \dot \q_\beta : \alpha \le \omega_3, \ \beta < \omega_3 \rangle$ 
so that for all $\alpha < \omega_3$, $\dot \q_\alpha$ is a $\p_\alpha$-name for a forcing which kills the 
stationarity of a nonreflecting 
subset of $\omega_2 \cap \cof(\omega)$, bookkeeping so that 
all nonreflecting stationary sets are handled. 
To prove that this forcing iteration is $\omega_2$-distributive, fix $\alpha < \omega_3$, 
and consider an appropriate 
elementary substructure $M$ containing $\p_\alpha$ 
with transitive collapsing map $\pi$. 
Then show that any condition 
in $M \cap \p_\alpha$ has an extension which 
lies in every dense open 
subset of $\p_\alpha$ in $M$.

The fact that $\p_\alpha$ is an iteration of natural posets 
adding clubs disjoint from nonreflecting subsets 
of $\omega_2$ implies that in 
$V[G \restrict (M \cap \kappa)]$, 
$\pi(\p_\alpha)$ is an iteration of natural posets adding 
clubs disjoint from \emph{nonstationary} subsets of $M \cap \kappa$. 
As such, $\pi(\p_\alpha)$ contains an $(M \cap \kappa)$-closed dense subset. 
It follows that the tail of the L\'evy collapse provides a 
$V[G \restrict (M \cap \kappa)]$-generic filter 
on $\pi(\p_\alpha)$ in $V[G]$, and the image of this filter 
under $\pi^{-1}$ is an $M$-generic filter on $\p_\alpha$. 
Hence, a lower bound of this filter, which does 
exist, 
is a member of every dense open subset of 
$\p_\alpha$ in $M$.

Let us compare these arguments with our situation. 
Instead of forcing with 
a L\'evy collapse, our preparation forcing 
will be a countable 
support iteration of proper forcings which is designed to collapse $\kappa$ to become $\omega_2$ 
and ensure 
the existence of sufficiently generic 
filters for certain forcings. 
Let $G$ be a generic filter for the preparation forcing. 
In $V[G]$, we define a 
suitable mixed support forcing iteration $\p$ which 
adds reals and clubs disjoint from nonreflecting sets. 

Consider an elementary substructure $M$ 
with transitive collapsing map $\pi$. 
In order to prove that $\p^*$ is $\omega_2$-distributive, 
one might try to argue similarly as above that in 
$V[G \restrict (M \cap \kappa)]$, 
$\pi(\p)$ is a suitable mixed support forcing iteration 
for adding reals and adding clubs disjoint from 
nonstationary sets. 
It turns out, however, that 
we can only show that the product 
$\pi(\p^c \otimes \p^*)$ forces that the collapse of a 
nonreflecting set is nonstationary, rather than $\pi(\p)$. 
Nonetheless, by some technical arguments this will suffice 
to prove that $\p^*$ is $\omega_2$-distributive, and hence 
that $\p$ preserves cardinals.

\begin{proposition}
	Let $\langle \p_\beta : \beta \le \alpha \rangle$ be 
	a suitable mixed support forcing iteration.
	Let $\beta \le \alpha$. 
	If $\p_\beta^*$ is $\omega_2$-distributive, 
	then $\p_\beta$ preserves $\omega_1$ and $\omega_2$.
	\end{proposition}

\begin{proof}
	Suppose for a contradiction that $p \in \p_\beta$ 
	forces that either $\omega_1^V$ or $\omega_2^V$ is no longer 
	a cardinal in $V^{\p_\beta}$. 
	Let $a := p \restrict \even$. 
	Let $H$ be a generic filter on $\p_\beta^c \otimes \p_\beta^*$ 
	which contains the condition $(a,p)$. 
	Let $G := \tau_\beta[H]$. 
	Then $G$ is a generic filter on $\p_\beta$ 
	by Lemma 1.18 and 
	$p = p + a = \tau_\beta(a, p)$ is in $G$. 
	Therefore, either 
	$\omega_1^V$ or $\omega_2^V$ is no longer a cardinal 
	in $V[G]$, and hence in $V[H]$.
	
	By Lemma 1.16, $V[H] = V[K]$, 
	where $K = K_1 \times K_2$ is a generic 
	filter on $(\p_\beta^c / a) \times 
	(\p_\beta^* / p)$. 
	Now $\p_\beta^*$ is $\omega_2$-distributive 
	by assumption, so $\omega_1^V$ and $\omega_2^V$ 
	remain cardinals in $V[K_2]$. 
	By absoluteness, in $V[K_2]$, 
	$\p_\beta^c$ is still 
	isomorphic to Cohen forcing, and hence 
	is $\omega_1$-c.c. 
	Therefore, 
	$\omega_1^V$ and $\omega_2^V$ remain cardinals in 
	$V[K_2][K_1] = V[K_1][K_2] = V[K] = V[H]$, which 
	is a contradiction.
\end{proof}

\begin{proposition}
	Assume that $2^{\omega_1} = \omega_2$. 
	Let $\langle \p_\beta : 
	\beta \le \omega_3 \rangle$ be 
	a suitable mixed support forcing iteration. 
	Suppose that for all $\beta < \omega_3$, 
	$\p_\beta^*$ is $\omega_2$-distributive. 
	Then $\p_{\omega_3}^*$ is $\omega_2$-distributive, 
	and hence preserves $\omega_1$ and $\omega_2$.
	\end{proposition}

\begin{proof}
	Let $\p := \p_{\omega_3}^*$. 
	Consider $p \in \p$. 
	Let $a := p \restrict \even$. 
	Then easily $p \in \p / a$.
	
	Suppose that $p$ forces in $\p$ that 
	$\{ \dot \alpha_i : i < \omega_1 \}$ 
	is a set of ordinals. 
	We will find $q$ below $p$ in $\p$ which decides 
	the value of $\dot \alpha_i$, for all 
	$i < \omega_1$. 
	Without loss of generality, we can assume that 
	each $\dot \alpha_i$ is a 
	nice $(\p / a)$-name for an ordinal. 
	It easily follows by Lemma 1.19(3) that 
	each $\dot \alpha_i$ 
	is a nice $(\p_\beta^* / a)$-name for an ordinal 
	for some $\beta < \omega_3$. 
	Thus, we can find an ordinal $\xi < \omega_3$ such that 
	$p \in \p_\xi^* / a$ and each $\dot \alpha_i$ 
	is a $(\p_\xi^* / a)$-name for an ordinal.
	
	Since $\p_\xi^*$ is $\omega_2$-distributive by 
	assumption, fix $q \le_\xi^* p$ 
	which decides in $\p_\xi^*$ the value of 
	$\dot \alpha_i$ for all $i < \omega_1$. 
	Then $q \le_\p p$ and $q$ decides in $\p$ the 
	value of $\dot \alpha_i$ for all $i < \omega_1$.
	\end{proof}

For the remainder of the section, fix a suitable mixed support forcing 
iteration $\langle \p_\beta : \beta \le \alpha \rangle$, where $\alpha < \omega_3$, 
based on a sequence of names $\langle \dot S_\gamma : \gamma \in \alpha \cap \odd \rangle$.

Before stating the next result, we make some clarifying 
remarks about names. 
Consider $\beta \le \alpha$. 
Then we have four forcing posets associated with $\beta$:
$\p_\beta^c$, $\p_\beta$, $\p_\beta^*$, and 
$\p_\beta^c \otimes \p_\beta^*$. 
If $H$ is a generic filter on 
$\p_\beta^c \otimes \p_\beta^*$, then 
$G := \tau_\beta[H]$ is a generic filter on 
$\p_\beta$, and in turn $G^c := G \cap \p_\beta^c$ 
is a generic filter on $\p_\beta^c$. 
Accordingly, if $\dot x$ is either a 
$\p_\beta$-name or a $\p_\beta^c$-name, when we talk about 
$\dot x$ in the context of statements in the forcing 
language of $\p_\beta^c \otimes \p_\beta^*$, we really 
mean the 
$(\p_\beta^c \otimes \p_\beta^*)$-name for the 
interpretation of $\dot x$ under 
$\tau_\beta[\dot H]$ or $\tau_\beta[\dot H] \cap \p_\beta^c$ 
respectively. 
Similar comments apply to $\p_\beta^c$-names 
in the context of the forcing language for $\p_\beta$.

The next two technical results will be crucial for the 
rest of the paper.

\begin{proposition}
	Let $\beta \le \alpha$, and assume that 
	$\p_\beta^*$ is $\omega_2$-distributive. 
	Suppose that $\dot x$ is a 
	$(\p_\beta^c \otimes \p_\beta^*)$-name for a set of 
	ordinals of size less than $\omega_2$. 
	Then for all $p \in \p_\beta$, there is $q \le_\beta^* p$ and a nice $\p_\beta^c$-name 
	$\dot b$ of size $\omega_1$ 
	such that $(q \restrict \even,q)$ forces in 
	$\p_\beta^c \otimes \p_\beta^*$ that $\dot x = \dot b$. 
	\end{proposition}

\begin{proof}
Let $\dot x$ and $p$ be as above. 
Let $a := p \restrict \even$. 
For the purpose of finding the condition $q$ and the name 
$\dot b$, let us consider 
a generic filter $H$ on $\p_\beta^c \otimes \p_\beta^*$ 
which contains the condition $(a,p)$. 
By Lemma 1.16, $V[H] = V[K]$, 
where $K = K_1 \times K_2$ is a generic 
filter on $(\p_\beta^c / a) \times 
(\p_\beta^* / p)$.

Let $x := \dot x^H$. 
Then $x \in V[K_2][K_1]$. 
Since $\p_\beta^c$ is still isomorphic to Cohen forcing 
in $V[K_2]$, 
we can cover $x$ by some set of ordinals $y \in V[K_2]$ 
of size $\omega_1$. 
Now fix in $V[K_2]$ a nice $(\p_\beta^c / a)$-name 
$\dot b$ for a subset of $y$ such that $\dot b^{K_1} = x$. 
Moreover, by the maximality principle applied in $V[K_2]$, 
we can find such a nice name 
so that $\p_\beta^c / a$ forces over $V[K_2]$ that 
$\dot b$ is equal to $\dot x$ (interpreted by the 
appropriate generic filters).

Since $\dot b$ is a nice name for a subset of $y$ and $\p_\beta^c / a$ is $\omega_1$-c.c.\! 
in $V[K_2]$, $\dot b$ has size $\omega_1$ in $V[K_2]$. 
Easily $\dot b \subseteq V$. 
Therefore, since $\p_\beta^*$ is $\omega_2$-distributive, 
the name $\dot b$ is in $V$. 
As $K_2$ is a $V$-generic filter on $\p_\beta^* / p$, 
we can find $q \le_\beta^* p$ in $K_2$ 
which forces in $\p_\beta^*$ that $\p_\beta^c / a$ 
forces that $\dot b$ equals $\dot x$. 
It is now straightforward to check that $q$ and 
$\dot b$ are as required.
\end{proof}

\begin{proposition}
	Let $\beta \le \alpha$, and assume that 
	$\p_\beta^*$ is $\omega_2$-distributive. 
	Suppose that $\dot x$ is a 
	$\p_\beta$-name for a set of 
	ordinals of size less than $\omega_2$. 
	Then for all $p \in \p_\beta$, there is $q \le_\beta^* p$ and a nice $\p_\beta^c$-name 
	$\dot b$ of size $\omega_1$ 
	such that $q$ forces in $\p_\beta$ that $\dot x = \dot b$.
	\end{proposition}

\begin{proof}
	Let $\dot x'$ be a $(\p_\beta^c \otimes \p_\beta^*)$-name for the interpretation 
	of $\dot x$ by $\tau_\beta[\dot H]$, where 
	$\dot H$ is the canonical name 
	for the generic filter on 
	$\p_\beta^c \otimes \p_\beta^*$. 
	Then obviously $\dot x'$ is a 
	$(\p_\beta^c \otimes \p_\beta^*)$-name for a set of 
	ordinals of size less than $\omega_2$. 
	By Proposition 2.3, there is $q \le_\beta^* p$ 
	and a nice $\p_\beta^c$-name 
	$\dot b$ of size $\omega_1$ 
	such that $(q \restrict \even,q)$ forces in 
	$\p_\beta^c \otimes \p_\beta^*$ that $\dot x' = \dot b$.
	
	It remains to show that $q$ forces in 
	$\p_\beta$ that $\dot x = \dot b$. 
	Suppose for a contradiction that 
	$r \le_\beta q$ and $r$ forces in $\p_\beta$ 
	that $\dot x \ne \dot b$. 
	Let $a := r \restrict \even$. 
	By Proposition 1.13, fix $r' \in \p_\beta$ 
	such that $r \le_\beta r' \le_\beta^* q$ 
	and $r \le_\beta r' + a \le_\beta r$. 
	Then $r'$ and $a$ are compatible in $\p_\beta$, 
	so $(a,r') \in \p_\beta^c \otimes \p_\beta^*$.
	
	Fix a generic filter $H$ on 
	$\p_\beta^c \otimes \p_\beta^*$ which contains $(a,r')$. 
	Note that $(a,r') \le (q \restrict \even,q)$, 
	so $(q \restrict \even,q) \in H$. 
	Let $G := \tau_\beta[H]$ and $G^c := G \cap \p_\beta^c$. 
	Then $\tau_\beta(a,r') = r' + a \in G$. 
	Since $r' + a \le_\beta r$, $r \in G$. 
	By the choice of $r$, 
	$\dot x^G \ne \dot b^{G^c}$. 
	By the choice of $q$, $(\dot x')^{H} = \dot b^{G^c}$. 
	Finally, by the choice of $\dot x'$, 
	$(\dot x')^H = \dot x^G$. 
	Thus, $\dot x^G \ne \dot b^{G^c}$ and yet 
	$\dot x^G = (\dot x')^H = \dot b^{G^c}$, which 
	is a contradiction.
\end{proof}

For a set $A \subseteq \omega_2$, let 
$\mathrm{CU}(A)$ denote the forcing poset consisting 
of closed and bounded subsets of $A$, ordered 
by end-extension. 
Assuming that $A$ is unbounded in $\omega_2$, it is easy 
to check that $\mathrm{CU}(A)$ adds a closed and cofinal 
subset of $\omega_2$ which is contained in $A$.

One of the main consequences of Proposition 2.4 is that 
our suitable mixed support forcing iteration will 
in fact add the desired 
generic filters for the club adding forcings.

\begin{proposition}
	Let $\gamma < \alpha$ be odd, and 
	assume that $\p_\gamma^*$ is $\omega_2$-distributive. 
	Then $\p_{\gamma+1}$ is forcing equivalent to 
	$\p_\gamma * \textrm{CU}(\omega_2 \setminus \dot S_\gamma)$.
\end{proposition}

\begin{proof}
	Let $\q := \p_\gamma * \textrm{CU}(\omega_2 \setminus \dot S_\gamma)$. 
	Define $f : \p_{\gamma+1} \to \q$ 
	by $f(p) := (p \restrict \gamma) * p(\gamma)$. 
	Let us check that $f$ actually maps into $\q$. 
	For a condition $p \in \p_{\gamma+1}$, 
	Definition 1.1(3) implies that 
	\begin{enumerate}
		\item $p \restrict \gamma \in \p_\gamma$;
		\item $p(\gamma)$ is a $\p_\gamma^c$-name for a closed and bounded subset of $\omega_2$;
		\item $p \restrict \gamma \Vdash_\gamma p(\gamma) \cap \dot S_\gamma = \emptyset$.
		\end{enumerate}
	By Lemma 1.12, (2) implies that $p(\gamma)$ is a $\p_\gamma$-name for a closed and 
	bounded subset of $\omega_2$. 
	So by (3), $p \restrict \gamma \Vdash_\gamma p(\gamma) \in \textrm{CU}(\omega_2 \setminus \dot S_\gamma)$.
	Hence, $f(p) = (p \restrict \gamma) * p(\gamma)$ is in $\q$.
	
	We claim that $f$ is a dense embedding. 
	It suffices to show that for all $p$ and $q$ in $\p_{\gamma+1}$, 
	$q \le_{\gamma+1} p$ iff $f(q) \le_\q f(p)$, and the range of $f$ is dense in $\q$.

	Consider $p$ and $q$ in $\p_{\gamma+1}$. 
	Then by Lemma 1.5(3), $q \le_{\gamma+1} p$ iff 
	\begin{itemize}
		\item[(a)] $q \restrict \gamma \le_\gamma p \restrict \gamma$;
		\item[(b)] $q \restrict (\gamma \cap \even)$ forces in $\p_\gamma^c$ that 
		$q(\gamma)$ end-extends $p(\gamma)$.
		\end{itemize}
	Assume that $q \le_{\gamma+1} p$. 
	Then $q \restrict \gamma \le_\gamma p \restrict \gamma$. 
	To see that $f(q) = (q \restrict \gamma) * q(\gamma) \le_\q 
	(p \restrict \gamma) * p(\gamma)$, it remains to show that 
	$q \restrict \gamma \Vdash_\gamma q(\gamma) \le_{\textrm{CU}(\omega_2 \setminus \dot S_\gamma)} p(\gamma)$, or  
	in other words, 
	that $q \restrict \gamma$ forces in $\p_\gamma$ that 
	$q(\gamma)$ end-extends $p(\gamma)$. 
	By Lemma 1.12, this follows from (b) above.
	
	Assume conversely that $f(q) \le_{\q} f(p)$. 
	Then $q \restrict \gamma \le_\gamma p \restrict \gamma$, and 
	$q \restrict \gamma$ forces in $\p_\gamma$ that 
	$q(\gamma) \le_{\textrm{CU}(\omega_2 \setminus \dot S_\gamma)} p(\gamma)$. 
	Hence, $q \restrict \gamma$ forces in $\p_\gamma$ that $q(\gamma)$ end-extends $p(\gamma)$. 
	By Lemma 1.12, $q \restrict (\gamma \cap \even)$ 
	forces in $\p_\gamma^c$ that 
	$q(\gamma)$ end-extends $p(\gamma)$. 
	By Lemma 1.5(3), $q \le_{\gamma+1} p$.
	
	To show that $f$ is dense, consider $r \in \q$. 
	Then $r = r_0 * \dot r_1$, where $r_0 \in \p_\gamma$ and $r_0$ forces in $\p_\gamma$ that 
	$\dot r_1 \in \textrm{CU}(\omega_2 \setminus \dot S_\gamma)$. 
	We will find $w \in \p_{\gamma+1}$ 
	such that $f(w) \le_\q r$. 
	By extending $r$ if necessary, we may 
	assume without loss of generality that 
	$r_0$ forces that $\dot r_1$ is nonempty.
	
	By Proposition 2.4, fix $t \le_\gamma^* r_0$ and a nice $\p_\gamma^c$-name 
	$\dot b$ such that $t \Vdash_\gamma \dot r_1 = \dot b$. 
	By the maximality principle for names, we may assume 
	that $\dot b$ is a nice $\p_\gamma^c$-name for a 
	nonempty closed and bounded subset of $\omega_2$. 
	We claim that $w := t \cup \{ ( \gamma, \dot b ) \}$ is a condition in $\p_{\gamma+1}$ 
	and $f(w) \le_\q r$. 
	We know that $w \restrict \gamma = t$ is in $\p_\gamma$, $w(\gamma) = \dot b$ is a nice $\p_\gamma^c$-name 
	for a nonempty closed and 
	bounded subset of $\omega_2$, and 
	$w \restrict \gamma = t$ forces in $\p_\gamma$ that $w(\gamma) = \dot b$ 
	is equal to $\dot r_1$, which is in 
	$\textrm{CU}(\omega_2 \setminus \dot S_\gamma)$ and 
	hence is disjoint from $\dot S_\gamma$. 
	By Definition 1.1, $w \in \p_{\gamma+1}$. 
	Since $t \le_\gamma^* r_0$, we have that 
	$t \le_\gamma r_0$. 
	Also, $t$ forces in $\p_\gamma$ that 
	$\dot r_1 = \dot b$, and hence obviously that 
	$\dot b \le \dot r_1$ in 
	$\textrm{CU}(\omega_2 \setminus \dot S_\gamma)$. 
	Therefore, $f(w) = t * \dot b$ extends 
	$r = r_0 * \dot r_1$ in $\q$. 
	\end{proof}

We now turn to studying conditions under which 
$\p_\alpha^*$ is $\omega_2$-distributive. 
The main result is Proposition 2.9 
below. 

\begin{lemma}
	Let $\gamma < \alpha$ be odd. 
	Assume that $\dot C$ is a 
	$(\p_\gamma^c \otimes \p_\gamma^*)$-name for a 
	club subset of $\omega_2$ which is disjoint from $\dot S_\gamma$. 
	Let $p \in \p_\gamma$ and $\dot \zeta$ be a $\p_\gamma$-name for an ordinal. 
	If $(p \restrict \even,p)$ forces in $\p_\gamma^c \otimes \p_\gamma^*$ that 
	$\dot \zeta$ is in $\dot C$, then $p$ forces in $\p_\gamma$ that 
	$\dot \zeta$ is not in $\dot S_\gamma$.
	\end{lemma}

\begin{proof}
	Suppose for a contradiction that there is $q \le_\gamma p$ which forces 
	in $\p_\gamma$ that $\dot \zeta$ is in $\dot S_\gamma$. 
	Let $b := q \restrict \even$. 
	Apply Proposition 1.13 to fix 
	$q' \in \p_\gamma$ such that 
	$q \le_\gamma q' \le_\gamma^* p$ 
	and $q \le_\gamma q' + b \le_\gamma q$.
	
	Let $H$ be a generic filter on 
	$\p_\gamma^c \otimes \p_\gamma^*$ which contains 
	the condition $(b,q')$. 
	Let $G := \tau_\gamma[H]$, which is a generic filter 
	on $\p_\gamma$. 
	Let $\zeta := \dot \zeta^G$, 
	$S_\gamma := \dot S_\gamma^G$, and 
	$C := \dot C^H$. 
	Then $C \cap S_\gamma = \emptyset$.

	Since $q' \le_\gamma^* p$ and 
	$b \le_\gamma^c p \restrict \even$, it follows that 
	$(b,q') \le (p \restrict \even,p)$, and hence 
	$(p \restrict \even,p) \in H$. 
	Therefore, $\zeta \in C$. 
	Since $C$ is disjoint from $S_\gamma$, 
	$\zeta \notin S_\gamma$. 
	On the other hand, 
	$\tau_\gamma(b,q') = q' + b \in G$ 
	and $q' + b \le_\gamma q$, so $q \in G$. 
	By the choice of $q$, $\zeta \in S_\gamma$, and 
	we have a contradiction.
	\end{proof}

\begin{notation}
	Let $\beta \le \alpha$. 
	Define the relation $\le_\beta^{*,s}$ on $\p_\beta$ by letting 
	$q \le_\beta^{*,s} p$ if for all $r \le_\beta^* q$, 
	$r$ and $p$ are compatible in $\p_\beta^*$. 
	We will abbreviate the forcing poset 
	$(\p_\beta,\le_\beta^{*,s})$ as $\p_\beta^{*,s}$.
	\end{notation}

Note that $q \le_\beta^* p$ implies that 
$q \le_\beta^{*,s} p$. 
It is easy to verify that 
the forcing poset $\p_\beta^{*,s}$ is separative, and the 
identity function is a dense embedding of $\p_\beta^*$ into 
$\p_\beta^{*,s}$.

\begin{lemma}
	Let $\beta \le \alpha$. 
	Assume that $q \le_\beta^{*,s} p$. Then:
	\begin{enumerate}
		\item $p \restrict \even = q \restrict \even$;
		\item $\dom(p) \subseteq \dom(q)$;
		\item for all odd $\gamma \in \dom(p)$, 
		$p \restrict (\gamma \cap \even)$ forces in $\p_\gamma^c$ 
		that one of $p(\gamma)$ and $q(\gamma)$ is an end-extension of the other. 
	\end{enumerate}
\end{lemma}

\begin{proof}
	(1) By the definition of $\le_\beta^{*,s}$, clearly $p$ and $q$ are 
	compatible in $\p_\beta^*$. 
	Fix $r \le_\beta^* p, q$. 
	Then 
	$p \restrict \even = r \restrict \even = q \restrict \even$.
	
	(2) If not, then by (1) we can fix an odd 
	ordinal 
	$\gamma \in \dom(p) \setminus \dom(q)$. 
	Fix a $\p_\gamma^c$-name $\dot a$ for the singleton consisting of 
	the least member of $\omega_2 \cap \cof(\omega_1)$ which is strictly larger than $\max(p(\gamma))$ 
	(we are using the fact that $p(\gamma)$ is forced to be nonempty by Definition 1.1(3)). 
	Clearly, $\p_\gamma^c$ forces that $\dot a$ and $p(\gamma)$ have no common 
	end-extension, and since $\p_\gamma$ forces that $\dot S_\gamma$ consists of ordinals 
	of cofinality $\omega$, $\p_\gamma$ forces that $\dot a$ is disjoint from $\dot S_\gamma$. 
	Define $s := q \cup \{ (\gamma, \dot a ) \}$. 
	Then $s \in \p_\beta$, $s \le_\beta^* q$, 
	and $s$ and $p$ are incompatible in $\p_\beta^*$. 
	This contradicts the assumption that $q \le_\beta^{*,s} p$.
	
	(3) Let $\gamma \in \dom(p) \cap \odd$. 
	Then by (2), $\gamma \in \dom(q)$. 
	Since $p$ and $q$ are compatible in $\p_\beta^*$, fix 
	$r \le_\beta^* p, q$. 
	As $\gamma \in \dom(p) \cap \dom(q)$, 
	$r \restrict (\gamma \cap \even)$ forces in $\p_\gamma^c$ that 
	$r(\gamma)$ is an end-extension of both $p(\gamma)$ and $q(\gamma)$. 
	In particular, it forces that $p(\gamma)$ and $q(\gamma)$ have a common end-extension, and 
	hence that one of them is an end-extension of the other. 
	But $r \le_\beta^* p$ implies that 
	$r \restrict \even = p \restrict \even$, 
	so $p \restrict (\gamma \cap \even)$ forces the same.
\end{proof}

\begin{proposition}
	Assume that for all odd $\gamma < \alpha$, $\p_\gamma^c \otimes \p_\gamma^*$ forces that 
	$\dot S_\gamma$ is a nonstationary subset of $\omega_2$. 
	Then both $\p_\alpha^*$ and $\p_\alpha^{*,s}$ 
	contain an $\omega_2$-closed dense subset. 
	\end{proposition}

\begin{proof}
	For each odd $\gamma < \alpha$, 
	fix a $(\p_\gamma^c \otimes \p_\gamma^*)$-name $\dot C_\gamma$ for a club subset 
	of $\omega_2$ which is disjoint from $\dot S_\gamma$. 
	For each $\beta \le \alpha$, define $D_\beta$ as the set of conditions 
	$p \in \p_\beta$ such that for all odd $\gamma \in \dom(p)$, 
	$(p \restrict (\gamma \cap \even),p \restrict \gamma)$ 
	forces in $\p_\gamma^c \otimes \p_\gamma^*$ that $\max(p(\gamma)) \in \dot C_\gamma$. 
	Observe that for all $\xi \le \beta \le \alpha$, $D_\xi \subseteq D_\beta$, 
	and if $p \in D_{\beta}$, 
	then $p \restrict \xi \in D_{\xi}$.
	
	We claim that for all $\beta \le \alpha$, 
	$D_\beta$ is an $\omega_2$-closed dense subset of both 
	$\p_\beta^*$ and $\p_\beta^{*,s}$. 
	The proof will be by induction on $\beta$, with the case $\beta = \alpha$ 
	concluding the proof of the proposition. 
	So fix $\beta \le \alpha$, and assume that for all $\xi < \beta$, 
	$D_{\xi}$ is an $\omega_2$-closed dense subset of both 
	$\p_\xi^*$ and $\p_{\xi}^{*,s}$. 
	It follows that for all $\xi < \beta$, the forcing poset $\p_\xi^*$ 
	is $\omega_2$-distributive, since it is forcing equivalent to an $\omega_2$-closed 
	forcing poset.
	
	We begin by proving closure. 
	We will show that any 
	$\le_\beta^{*,s}$-descending sequence of conditions in 
	$D_\beta$ of length a limit 
	ordinal less than $\omega_2$ has a $\le_\beta^*$-lower 
	bound in $D_\beta$. 
	Note that this implies that $D_\beta$ is $\omega_2$-closed 
	in both $\p_\beta^*$ and $\p_\beta^{*,s}$. 
	So consider a $\le_\beta^{*,s}$-descending sequence 
	$\langle p_i : i < \delta \rangle$ of conditions in $D_\beta$, where $\delta < \omega_2$ 
	is a limit ordinal. 
	We will find $q \in D_\beta$ such that 
	$q \le_\beta^{*} p_i$ for all $i < \delta$. 
	Let $a := p_0 \restrict \even$. 
	Then by Lemma 2.8(1), 
	for all $i < \delta$, $p_i \restrict \even = a$.
	
	Define $q$ as follows. 
	Let $q \restrict \even := a$. 
	Let $\dom(q) \cap \odd := \bigcup \{ \dom(p_i) \cap \odd : i < \delta \}$. 
	Consider an odd ordinal $\gamma$ in $\dom(q)$. 
	By Lemma 2.8(3), $a \restrict \gamma$ forces in 
	$\p_\gamma^c$ that 
	$\{ p_i(\gamma) : i < \delta \}$ is a family 
	of closed and bounded subsets of $\omega_2$ which are 
	pairwise comparable under end-extension. 
	It easily follows that $a \restrict \gamma$ forces 
	that the union of this family is bounded in $\omega_2$ 
	and is closed below its supremum. 
	Let $q(\gamma)$ be a nice $\p_\gamma^c$-name for a 
	nonempty closed and bounded subset of 
	$\omega_2$ which, if $a \restrict \gamma$ is in 
	the generic filter on $\p_\gamma^c$, is equal to 
	the union of $\{ p_i(\gamma) : i < \delta \}$ together with the ordinal 
	$\sup \{ \max(p_i(\gamma)) : i < \delta \}$.
	
	We prove by induction on $\xi \le \beta$ that 
	$q \restrict \xi \in D_\xi$ and 
	$q \restrict \xi \le_\xi^* p_i \restrict \xi$ for all $i < \delta$. 
	It then follows that $q \in D_\beta$ 
	and $q \le_\beta^* p_i$ for all $i < \delta$. 
	Referring to Definition 1.1, the only nontrivial case to consider is when 
	$\xi = \gamma + 1$ for an odd ordinal $\gamma$.
	
	So assume that $\gamma < \beta$ is odd and $q \restrict \gamma$ is as required. 
	Then $q \restrict \gamma \le_\gamma^* p_i \restrict \gamma$ for all $i < \delta$. 
	By the definition of $D_\beta$, 
	each $p_i$ with $\gamma \in \dom(p_i)$ 
	satisfies that $(p_i \restrict (\gamma \cap \even),p_i \restrict \gamma) 
	= (a \restrict \gamma,p_i \restrict \gamma)$ 
	forces in $\p_\gamma^c \otimes \p_\gamma^*$ that $\max(p_i(\gamma)) \in \dot C_\gamma$. 
	The fact that 
	$q \restrict \gamma \le_\gamma^* p_i \restrict \gamma$ implies that 
	$(q \restrict (\gamma \cap \even),q \restrict \gamma) = 
	(a \restrict \gamma,q \restrict \gamma)$ is below 
	$(a \restrict \gamma,p_i \restrict \gamma)$ in 
	$\p_\gamma^c \otimes \p_\gamma^*$. 
	Therefore, $(q \restrict (\gamma \cap \even),q \restrict \gamma)$ forces in 
	$\p_\gamma^c \otimes \p_\gamma^*$ that $\max(p_i(\gamma)) \in \dot C_\gamma$.
	
	Since the above is true for all $i < \delta$ and $\dot C_\gamma$ is a name 
	for a club, it follows that 
	$(q \restrict (\gamma \cap \even),q \restrict \gamma)$ forces in 
	$\p_\gamma^c \otimes \p_\gamma^*$ that 
	$\sup \{ \max(p_i(\gamma)) : i < \delta \} = \max(q(\gamma)) \in \dot C_\gamma$. 
	By Lemma 2.6, $q \restrict \gamma$ forces in $\p_\gamma$ that 
	$\max(q(\gamma)) \notin \dot S_\gamma$. 
	Since $q \restrict \gamma$ forces that 
	any other member of $q(\gamma)$ is in $p_i(\gamma)$ 
	for some $i < \delta$, and 
	$q \restrict \gamma \le_\gamma 
	p_i \restrict \gamma$ for all $i < \delta$, 
	it follows that $q \restrict \gamma$ forces that 
	$q(\gamma)$ is disjoint from $\dot S_\gamma$. 
	Thus, $q \restrict (\gamma+1)$ is 
	in $\p_{\gamma+1}$. 
	Now the inductive hypothesis and the above 
	arguments imply that 
	$q \restrict (\gamma+1) \in D_{\gamma+1}$ and 
	$q \restrict (\gamma+1) \le_{\gamma+1}^* p_i \restrict (\gamma+1)$ for all $i < \delta$. 
	This completes the proof of closure.

	It remains to show that $D_\beta$ is a dense subset of 
	$\p_\beta^*$ and $\p_\beta^{*,s}$. 
	Note that it suffices to prove that $D_\beta$ is 
	dense in $\p_\beta^*$. 
	Consider $p \in \p_\beta^*$, and we will find 
	$q \le_\beta^* p$ in $D_\beta$. 
	First, assume that $\beta = \xi + 1$ is a successor ordinal. 
	If $\xi$ is even, then fix 
	$q_0 \le_\xi^* p \restrict \xi$ in $D_\xi$ by the inductive hypothesis. 
	Then $q_0 \cup \{ (\xi,p(\xi))\} \le_\beta^* p$ is in $D_{\beta}$. 
	Now suppose that $\xi$ is odd. 
	If $\xi \notin \dom(p)$, then 
	fix $q_0 \le_\xi^* p \restrict \xi$ in $D_\xi$ by the inductive hypothesis. 
	Then $q_0 \le_\beta^* p$ and 
	$q_0 \in D_{\beta}$.
	
	Suppose that $\xi \in \dom(p)$. 
	Let $\dot x$ be a $(\p_\xi^c \otimes \p_\xi^*)$-name for 
	$p(\xi)$ together with the least member of $\dot C_\xi$ strictly above $\max(p(\xi))$. 
	Since $\p_\xi^*$ is $\omega_2$-distributive by the inductive hypothesis, 
	by Proposition 2.3 we can fix 
	$q_0 \le_\xi^* p \restrict \xi$ and a nice 
	$\p_\xi^c$-name $\dot b$ such that 
	$(q_0 \restrict \even,q_0)$ forces in 
	$\p_\xi^c \otimes \p_\xi^*$ 
	that $\dot b = \dot x$. 
	By the maximality principle for names, we may assume without loss of generality 
	that $\dot b$ is a nice $\p_\xi^c$-name for a nonempty closed and bounded subset of $\omega_2$. 
	Note that $(q_0 \restrict \even,q_0)$ forces in 
	$\p_\xi^c \otimes \p_\xi^*$ 
	that $\max(\dot b) = \max(\dot x) \in \dot C_\xi$. 
	By Lemma 2.6, $q_0$ forces in $\p_\xi$ that 
	$\max(\dot b) \notin \dot S_\xi$. 
	Now fix $r_0 \le_\xi^* q_0$ in $D_\xi$ by the inductive hypothesis. 
	Let $r := r_0 \cup \{ (\xi,\dot b) \}$. 
	Since $r_0 \le_\xi q_0$, $r_0$ forces in 
	$\p_\xi$ that $\max(\dot b) \notin \dot S_\xi$. 
	As $r_0 \le_\xi p \restrict \xi$, $r_0$ forces 
	in $\p_\xi$ that 
	$\dot b$ is disjoint from $\dot S_\xi$. 
	Thus, $r \in \p_{\beta}$. 
	Also, clearly $r$ is in 
	$D_{\beta}$ and $r \le_\beta^* p$.
	
	Secondly, assume 
	that $\beta$ is a limit ordinal. 
	If $\cf(\beta) \ge \omega_2$, then  
	for some $\xi < \beta$, $\dom(p) \subseteq \xi$, and hence 
	$p \in \p_\xi$. 
	By the inductive hypothesis, we can fix $q \le_\xi^* p$ in $D_\xi$. 
	Then $q \le_\beta^* p$ is in $D_\beta$. 
	
	Suppose that $\cf(\beta) < \omega_2$. 
	Fix a strictly increasing and continuous sequence 
	$\langle \beta_i : i < \cf(\beta) \rangle$ 
	which is cofinal in $\beta$, and let 
	$\beta_{\cf(\beta)} = \beta$. 
	Since $\dom(p) \cap \even$ is finite, we may assume that 
	$\dom(p) \cap \even \subseteq \beta_0$. 
	We define by induction a $\le_\beta^*$-descending sequence of conditions 
	$\langle p_i : i \le \cf(\beta) \rangle$ below $p$ such that for each $i \le \cf(\beta)$,  
	$p_i \restrict \beta_i \in D_{\beta_i}$ if $i > 0$, 
	and $p_i \restrict [\beta_i,\beta) = p \restrict [\beta_i,\beta)$.
		
	Let $p_0 := p$. 
	Let $i < \cf(\beta)$, and assume that $p_j$ is defined as required for all $j \le i$. 
	By the inductive hypothesis, fix 
	$p_{i+1}^{-} \le_{\beta_{i+1}}^* p_i \restrict \beta_{i+1}$ 
	in $D_{\beta_{i+1}}$. 
	Now let $p_{i+1} := p_{i+1}^- \cup p \restrict [\beta_{i+1},\beta)$. 
	Then easily $p_{i+1}$ is as required.
	
	Let $\delta \le \cf(\beta)$ be a limit ordinal, and assume that $p_i$ is defined as 
	required for all $i < \delta$. 
	Then for all $i < j < \delta$, $p_j \le_\beta^* p_i$. 
	Since $\dom(p) \cap \even \subseteq \beta_0$, it easily follows that 
	for all $i < j < \delta$, 
	$p_j \restrict \beta_j \le_{\beta_\delta}^* p_i \restrict \beta_i$. 
	Therefore, $\langle p_i \restrict \beta_i : i < \delta \rangle$ 
	is a $\le_{\beta_\delta}^*$-descending sequence in $D_{\beta_\delta}$. 
	Since we have already proven the $\omega_2$-closure of $D_{\beta_\delta}$, 
	we can find $p_{\delta}^- \in D_{\beta_\delta}$ 
	such that $p_{\delta}^- \le_{\beta_\delta}^* p_i \restrict \beta_i$ for all 
	$i < \delta$. 
	As $\sup_{i < \delta} \beta_i = \beta_{\delta}$, 
	it easily follows that $p_{\delta}^- \le_{\beta_\delta}^* p_i \restrict \beta_\delta$ 
	for all $i < \delta$. 
	Let $p_{\delta} := p_{\delta}^- \cup p \restrict [\beta_\delta,\beta)$. 
	Then $p_{\delta} \le_\beta^* p_i$ for all $i < \delta$ and 
	$p_\delta \restrict \beta_\delta = p_{\delta}^- \in D_{\beta_\delta}$.
\end{proof}

The next result describes how we will use the 
preparation forcing in the proof of the main 
consistency result.

\begin{lemma}
	Assume that $2^{\omega_1} = \omega_2$. 
	Suppose that the forcing poset 
	$\p_\alpha^{*,s}$ contains an 
	$\omega_2$-closed dense subset. 
	Let $G \times H$ be a generic filter on $\add(\omega,\omega_2) \times \add(\omega_2)$. 
	Then in $V[G \times H]$, for any condition $(a,p) \in \p_\alpha^c \otimes \p_\alpha^*$ 
	such that $a \le_\alpha^c p \restrict \even$, 
	there exists a generic filter $K$ on 
	$\p_\alpha^c \otimes \p_\alpha^*$ 
	which contains $(a,p)$, and moreover, 
	$V[G \times H]$ is a generic extension of 
	$V[K]$ by an $\omega_1$-c.c.\! forcing poset.
	\end{lemma}

\begin{proof}
	Fix an $\omega_2$-closed dense subset $D$ of $\p_\alpha^{*,s}$. 
	Consider a condition 
	$(a,p) \in \p_\alpha^c \otimes \p_\alpha^*$ such that 
	$a \le_\alpha^c p \restrict \even$. 	
	Let $D_p := \{ q \in D : q \le_\alpha^{*,s} p \}$. 
	Then clearly $D_p$ is an $\omega_2$-closed dense 
	subset of $\p_\alpha^{*,s} / p$. 
	Since $\p_\alpha^{*,s}$ is separative, 
	obviously $(D_p,\le_\alpha^{*,s})$ is also separative, 
	and since $\p_\alpha$ has size $\omega_2$, so does $D_p$. 
	By standard forcing facts, it follows that 
	$(D_p,\le_\alpha^{*,s})$ is forcing equivalent to $\add(\omega_2)$.
	
	We also know by Lemma 1.3 
	that $\p_\alpha^c$ is isomorphic to $\add(\omega,\ot(\alpha \cap \even))$. 
	Since $\alpha < \omega_3$, $\p_\alpha^c$ is isomorphic to a regular suborder of 
	$\add(\omega,\omega_2)$ of the form 
	$\add(\omega,\delta)$ for some $\delta \le \omega_2$. 
	By standard facts, for any $s \in \add(\omega,\delta)$, 
	$\add(\omega,\delta) / s$ is isomorphic 
	to $\add(\omega,\delta)$. 
	Hence, $\p_\alpha^c / a$ is isomorphic to 
	$\add(\omega,\delta)$. 
	So $\add(\omega,\omega_2)$ is isomorphic to 
	$(\p_\alpha^c / a) \times 
	\add(\omega,\omega_2 \setminus \delta)$.
	
	From these facts, we can obtain in $V[H]$ a $V$-generic filter $H_1$ on $(D_p,\le_{\alpha}^{*,s})$ 
	such that $V[H] = V[H_1]$, and in $V[H][G]$ 
	we can obtain a $V[H]$-generic filter 
	$H_2$ on $\p_\alpha^c / a$ such that 
	$V[G \times H] = V[H][G]$ 
	is a generic extension of $V[H][H_2]$ by 
	the $\omega_1$-c.c.\! forcing $\add(\omega,\omega_2 \setminus \delta)$.
	
	Now the upwards closure $H_1'$ of $H_1$ in 
	$\p_\alpha^{*,s}$ 
	is a $V$-generic filter on $\p_\alpha^{*,s}$ 
	which contains $p$, and 
	$V[H] = V[H_1] = V[H_1']$. 
	Since the identity function is a dense embedding of 
	$\p_\alpha^*$ into $\p_\alpha^{*,s}$, 
	$H_1'$ is also a $V$-generic filter on $\p_\alpha^*$ 
	which contains $p$. 
	So $H_1' / p$ is a $V$-generic filter on 
	$\p_\alpha^* / p$ and 
	$V[H] = V[H_1'] = V[H_1' / p]$. 
	Thus, $H_2 \times (H_1' / p)$ is a $V$-generic filter on 
	$(\p_\alpha^c / a) \times (\p_\alpha^* / p) = 
	(\p_\alpha^c \otimes \p_\alpha^*) / (a,p)$. 
	Letting $K$ be the upwards closure of this filter in 
	$\p_\alpha^c \otimes \p_\alpha^*$, 
	$K$ is a generic filter on 
	$\p_\alpha^c \otimes \p_\alpha^*$ 
	which contains $(a,p)$, and 
	$V[K] = V[H_2 \times (H_1' / p)] = V[H][H_2]$. 
	And from the above, $V[G \times H]$ is a generic extension of 
	$V[H][H_2] = V[K]$ by an $\omega_1$-c.c.\! forcing poset.
	\end{proof}

We need one more lemma before proceeding to the main 
result of the paper.

\begin{lemma}
	Assume that for all $\beta < \alpha$, 
	$\p_\beta$ preserves $\omega_1$. 
	Suppose that $\langle p_i : i < \delta \rangle$ 
	is a $\le_\alpha^*$-descending sequence of conditions, 
	where $\delta \in \omega_2 \cap \cof(\omega_1)$. 
	Then there is $q$ such that $q \le_\alpha^* p_i$ 
	for all $i < \delta$.
	\end{lemma}

\begin{proof}
	Let $a := p_0 \restrict \even$. 
	Then $a = p_i \restrict \even$ for all $i < \delta$. 
	Define $q$ as follows. 
	Let $q \restrict \even = a$ and 
	$\dom(q) \cap \odd := \bigcup \{ \dom(p_i) \cap \odd : 
	i < \delta \}$. 
	For each odd $\gamma \in \dom(q)$, let 
	$q(\gamma)$ be a $\p_\gamma^c$-name for a nonempty 
	closed and bounded subset of $\omega_2$ such that, 
	assuming $a \restrict \gamma$ is in the generic filter, 
	then $q(\gamma)$ is the union of 
	$\{ p_i(\gamma) : i < \delta \}$ together with 
	the supremum of 
	$\{ \max(p_i(\gamma)) : i < \delta \}$.
	
	To see that $q$ is a condition, it suffices to show 
	that for all odd $\gamma < \alpha$, assuming 
	that $q \restrict \gamma$ is in $\p_\gamma$ and 
	is $\le_\gamma^*$-below $p_i \restrict \gamma$ for all 
	$i < \delta$, then $q \restrict \gamma$ forces in 
	$\p_\gamma$ that $\max(q(\gamma)) \notin \dot S_\gamma$. 
	But since $\delta$ has cofinality $\omega_1$, 
	$a \restrict \gamma$ forces that 
	$\max(q(\gamma))$ has cofinality 
	$\omega_1$, or for some $i < \delta$, 
	$\max(q(\gamma)) = \max(p_j(\gamma))$ for all 
	$i \le j < \delta$. 
	As $\dot S_\gamma$ is a 
	$\p_\gamma$-name for a subset of 
	$\omega_2 \cap \cof(\omega)$ 
	and $\p_\gamma$ preserves $\omega_1$, 
	in either case 
	$q \restrict \gamma$ forces that $\max(q(\gamma))$ is 
	not in $\dot S_\gamma$. 
	\end{proof}

\section{The Consistency Result}

Let $\kappa$ be a Mahlo cardinal and assume that \textsf{GCH} holds. 
For example, if $\kappa$ is Mahlo, then $\kappa$ is Mahlo in $L$, so we can take 
our ground model to be $L$.
We will prove that there exists 
a forcing poset which collapses $\kappa$ to become $\omega_2$, forces that 
$2^\omega = \omega_3$, and forces that every stationary subset of $\omega_2 \cap \cof(\omega)$ 
reflects to an ordinal in $\omega_2$ with cofinality $\omega_1$. 
The forcing poset will be of the form $\R_\kappa * \p_{\kappa^+}$, where $\R_\kappa$ 
is a preparation forcing which collapses $\kappa$ to become $\omega_2$ 
and $\p_{\kappa^+}$ is a suitable mixed support forcing iteration in $V^{\R_\kappa}$ for killing 
nonreflecting sets.

To begin, let us define in the ground model $V$ 
a countable support forcing iteration 
$$
\langle \R_\alpha, \dot \s_\beta : \alpha \le \kappa, \ \beta < \kappa \rangle
$$
of proper forcings as follows. 
Let $\alpha < \kappa$, and assume that $\R_\beta$ and $\dot \s_\gamma$ are 
defined for all $\beta \le \alpha$ and $\gamma < \alpha$. 
If $\alpha$ is not inaccessible, then let $\dot \s_\alpha$ be an $\R_\alpha$-name 
for the collapse $\col(\omega_1,\omega_2)$. 
Then $\R_\alpha$ forces that 
$\dot \s_\alpha$ is $\omega_1$-closed, and hence proper. 
Let $\R_{\alpha+1} := \R_\alpha * \dot \s_\alpha$.

Now assume that $\alpha$ is inaccessible. 
Also, assume as a recursion hypothesis that 
$\R_\alpha$ is $\alpha$-c.c., 
has size $\alpha$, and collapses $\alpha$ to become $\omega_2$. 
Let $\dot \s_\alpha$ be an $\R_\alpha$-name for 
$\add(\omega,\omega_2) \times \add(\omega_2)$ (in other words, 
$\add(\omega,\alpha) \times \add(\alpha)$). 
Note that this product is forcing equivalent to the two-step forcing iteration 
$\add(\omega_2) * \add(\omega,\omega_2)$, which is an $\omega_1$-closed forcing 
followed by an $\omega_1$-c.c.\! forcing, and hence is proper. 
Let $\R_{\alpha+1} := \R_\alpha * \dot \s_\alpha$.

Now let $\delta \le \kappa$ be a limit ordinal, and 
assume that 
$\R_\beta$ and $\dot \s_\beta$ are defined for all $\beta < \delta$. 
Let $\R_\delta$ be the 
countable support limit of 
$\langle \R_\alpha : \alpha < \delta \rangle$. 
By standard arguments, it is easy to check that if $\delta$ is inaccessible, then 
the recursion hypothesis stated in the inaccessible case above holds for $\R_\delta$.

This completes the definition.  
The iteration $\R_\kappa$ is proper, $\kappa$-c.c.\!, and has size $\kappa$.  
So $\R_\kappa$ preserves $\omega_1$ and collapses $\kappa$ to become $\omega_2$. 
Standard nice name arguments show that $\R_\kappa$ 
forces that $2^\omega = 2^{\omega_1} = \omega_2$ and 
$2^\mu = \mu^+$ for all cardinals $\mu \ge \kappa$.

\bigskip

Let $G$ be a generic filter on $\R_\kappa$. 
In $V[G]$, we define a sequence of forcing posets  
$\langle \p_\beta : \beta \le \kappa^+ \rangle$. 
This sequence will be a suitable mixed support forcing 
iteration based on a sequence of names 
$\langle \dot S_\gamma : \gamma \in \kappa^+ \cap \odd \rangle$. 
Definition 1.1 provides a recursive 
description which will determine the iteration, provided that we 
specify the names $\dot S_\gamma$ for all $\gamma \in \kappa^+ \cap \odd$. 
Each name $\dot S_\gamma$ will be a nice $\p_\gamma$-name for a subset 
of $\omega_2 \cap \cof(\omega)$ such that 
$\p_\gamma$ forces that $\dot S_\gamma$ does not reflect to any ordinal 
in $\omega_2 \cap \cof(\omega_1)$.

We will assume two recursion hypotheses in $V[G]$. 
Let $\beta < \kappa^+$, and suppose that $\langle \p_\delta : \delta \le \beta \rangle$ and 
$\langle \dot S_\gamma : \gamma \in \beta \cap \odd \rangle$ are defined. 
The first recursion hypothesis is:

\begin{recursionhyp}
	For all $\xi \le \beta$, 
	the forcing poset $\p_{\xi}^*$ is $\omega_2$-distributive, and therefore $\p_\xi$ 
	preserves $\omega_1$ and $\omega_2$.
	\end{recursionhyp}

Let us see how we can prove the consistency result assuming that this first recursion hypothesis 
holds for all $\beta < \kappa^+$. 
By Lemma 1.19(2) and Proposition 2.2, 
$\p_{\kappa^+}$ is $\kappa^+$-c.c.\! 
and preserves 
$\omega_1$ and $\omega_2$. 
It easily follows that any nice $\p_{\kappa^+}$-name for a subset of $\kappa \cap \cof(\omega)$ which 
does not reflect to any ordinal of uncountable cofinality in $\kappa$ 
is also a nice $\p_\beta$-name for a set of the same kind for some $\beta < \kappa^+$. 
Since $\p_\beta$ has size $\kappa$ and $2^\kappa = \kappa^+$, 
after we define $\p_\beta$ we can enumerate all such $\p_\beta$-names 
in order type $\kappa^+$. 
When we select the names $\dot S_\gamma$, we use a standard bookkeeping function 
argument to arrange that any such name is equal to 
$\dot S_\gamma$ for some $\gamma < \kappa^+$. 
Since $\p_{\gamma+1}$ is a regular suborder of $\p_{\kappa^+}$ 
and is forcing equivalent to $\p_\gamma * \textrm{CU}(\kappa \setminus \dot S_\gamma)$ 
by Proposition 2.5, 
this nonreflecting set will become nonstationary after forcing with $\p_{\kappa^+}$. 
Thus, in the model $V^{\R_\kappa * \p_{\kappa^+}}$, 
every stationary subset of $\omega_2 \cap \cof(\omega)$ reflects to an 
ordinal in $\omega_2$ with cofinality $\omega_1$. 
Since $\p_{\kappa^+}$ adds $\kappa^+$ many reals, standard arguments show that 
in this final model, $2^\omega = \omega_3$.

In order to maintain the first recursion hypothesis, we will need a second more technical 
recursion hypothesis. 
Before stating it, we introduce some terminology.

\begin{notation}
A set $N$ in the ground model $V$ 
is said to be \emph{suitable} if 
$N$ is an elementary substructure of $H(\kappa^+)$ of size less than $\kappa$, 
$\kappa_N := N \cap \kappa$ is inaccessible, 
$|N| = \kappa_N$, 
$N^{< \kappa_N} \subseteq N$, and the forcing 
iteration $\vec \R := \langle \R_\alpha, \dot \s_\delta : \alpha \le \kappa, \ \delta < \kappa \rangle$ 
is a member of $N$.
	\end{notation}

The fact that $\kappa$ is Mahlo implies by standard 
arguments that there are stationarily many 
suitable sets in $P_{\kappa}(H(\kappa^+))$. 
The same comment applies regarding Notation 3.4 below.

\begin{lemma}
Suppose that $N$ is suitable. 
Let $\pi_N : N \to N_0$ be the transitive collapse of $N$. 
Let $\pi_{N[G]} : N[G] \to M_0$ be the transitive collapse of $N[G]$ in $V[G]$. 
Then:
\begin{enumerate}
	\item $\pi_N(\vec \R) = 
	\langle \R_\alpha, \dot \s_\delta : \alpha \le \kappa_N, \ \delta < \kappa_N \rangle$;
	\item $M_0 = N_0[G \restrict \kappa_N]$, and therefore $M_0 \in V[G \restrict \kappa_N]$;
	\item $\pi_{N[G]} \restrict N = \pi_N$.
\end{enumerate}
\end{lemma}

The proof is straightforward.

\begin{notation}
A set $N$ is said to be \emph{$\beta$-suitable} if $N$ is suitable and $N$ contains 
$\R_\kappa$-names for the objects 
$\langle \p_i : i \le \beta \rangle$ and 
$\langle \dot S_\gamma : \gamma \in \beta \cap \odd \rangle$.
\end{notation}

Observe that if $N$ is $\beta$-suitable, 
then for all $\beta' \in N \cap \beta$, 
$N$ is $\beta'$-suitable.

\begin{lemma}
	Let $N$ be $\beta$-suitable, $\pi_N : N \to N_0$ the transitive collapse of $N$, 
	and $\pi : N[G] \to N_0[G \restrict \kappa_N]$ the transitive collapse of $N[G]$. 
	Then in $V[G \restrict \kappa_N]$, 
	$\langle \p_i^\pi : i \le \pi(\beta) \rangle := \pi(\langle \p_i : i \le \beta \rangle)$ 
	is a suitable mixed support forcing iteration 
	based on the sequence of names 
	$\langle \dot S_\gamma^\pi : \gamma \in \pi(\beta) \cap \odd \rangle := 
	\pi(\langle \dot S_\gamma : \gamma \in \beta \cap \odd \rangle)$. 
	Moreover, $\pi(\p_\beta^c) = (\p_{\pi(\beta)}^\pi)^c$, 
	$\pi(\p_\beta^*) = (\p_{\pi(\beta)}^\pi)^*$, 
	$\pi(\p_\beta^c \otimes \p_\beta^*) = (\p_{\pi(\beta)}^\pi)^c \otimes (\p_{\pi(\beta)}^\pi)^*$ , 
	and $\pi(\p_\beta^{*,s}) = (\p_{\pi(\beta)}^\pi)^{*,s}$.
	\end{lemma}

\begin{proof}
	Let $M := N_0[G \restrict \kappa_N]$. 
	Then $\kappa_N = \pi(\kappa) \in M$ and 
	$\kappa_N$ equals $\omega_2$ in $V[G \restrict \kappa_N]$. 
	Since $N^{<\kappa_N} \subseteq N$ 
	and $N$ and $N_0$ 
	are isomorphic, $N_0^{<\kappa_N} \subseteq N_0$. 
	As $\R_{\kappa_N}$ is $\kappa_N$-c.c.\!, it follows 
	by standard facts that $M = N_0[G \restrict \kappa_N]$ is 
	closed under sequences of length less than 
	$\kappa_N$ in $V[G \restrict \kappa_N]$. 
	In particular, $M^{\omega_1} \subseteq M$ in 
	$V[G \restrict \kappa_N]$. 
	Since $M$ is isomorphic to $N[G]$, which is a model 
	of $\textsf{ZFC} - \textsf{Powerset}$, 
	$M$ is a model of $\textsf{ZFC} - \textsf{Powerset}$.

	Using absoluteness, 
	$\pi(\langle \p_i : i \le \beta \rangle)$ is a sequence 
	of forcing posets $\langle \p_i^\pi : i \le \pi(\beta) \rangle$, 
	and $\pi(\langle \dot S_\gamma : \gamma \in \beta \cap \odd \rangle)$ 
	is a sequence $\langle \dot S_\gamma^\pi : \gamma \in \pi(\beta) \cap \odd \rangle$ 
	such that for each $\gamma \in \pi(\beta) \cap \odd$, $\dot S_\gamma^\pi$ is a nice $\p_\gamma^\pi$-name 
	for a subset of $\kappa_N \cap \cof(\omega)$.
	
	Since $\pi$ is an isomorphism, $M$ models that $\langle \p_i^\pi : i \le \pi(\beta) \rangle$ 
	is a suitable mixed support forcing iteration based 
	on the sequence of names $\langle \dot S_\gamma^\pi : \gamma \in \pi(\beta) \cap \odd \rangle$. 
	By Lemma 1.2, it follows that in $V[G \restrict \kappa_N]$, 
	$\langle \p_i^\pi : i \le \pi(\beta) \rangle$ is a suitable mixed support forcing iteration based 
	on the sequence of names 
	$\langle \dot S_\gamma^\pi : \gamma \in \pi(\beta) \cap \odd \rangle$. 
	The remaining statements are easy to verify.
	\end{proof}

We are now ready to state the second recursion hypothesis.

\begin{recursionhyp}
	Let $N$ be $\beta$-suitable and $\pi$ be the 
	transitive collapsing map of $N[G]$. 
	Then for all odd $\gamma \in N \cap \beta$, in the model 
	$V[G \restrict \kappa_N]$, 
	$\pi(\p_\gamma^c \otimes \p_\gamma^*)$ 
	forces that $\pi(\dot S_\gamma)$ is a nonstationary subset of $\kappa_N$. 
	\end{recursionhyp}

It remains to prove that the two recursion hypotheses hold for all $\beta < \kappa^+$. 
The proof will proceed as follows. 
For a fixed $\beta < \kappa^+$, we will 
assume that the recursion hypotheses hold 
for all $\gamma \le \beta$, 
and then prove that they hold for $\beta+1$ by first verifying 
the second recursion hypothesis for $\beta+1$, 
and then using that hypothesis to 
prove the first recursion hypothesis for $\beta+1$. 
Then, for a fixed limit ordinal $\alpha < \kappa^+$, we will assume that both recursion 
hypotheses hold for all $\beta < \alpha$. 
Observe that the second recursion hypothesis then holds immediately for $\alpha$. 
So in the limit case it will suffice to prove the first recursion hypothesis for $\alpha$.

The proof of the first recursion hypothesis 
is the same for both successor and limit stages. 
Observe that if the second recursion hypothesis holds for $\beta$, where 
$\beta$ is even, then it immediately holds for $\beta+1$. 
Putting it all together, it will suffice to prove 
the second recursion hypothesis only in the successor case 
$\beta+1$ where $\beta$ is odd, and then 
prove the first recursion hypothesis in an independent way.

The proofs of both recursion hypotheses will use the following lemma.

\begin{lemma}
	Assume that both recursion hypotheses hold 
	for all $\gamma < \beta$ and 
	the second recursion hypothesis holds for $\beta$. 
	Let $N$ be $\beta$-suitable and 
	$(a,p) \in \p_\beta^c \otimes \p_\beta^*$. 
	Let $\pi$ be the transitive collapsing map of $N[G]$. 
	Then in $V[G]$ there exists a 
	$V[G \restrict \kappa_N]$-generic filter $K$ on 
	$\pi(\p_\beta^c \otimes \p_\beta^*)$ which contains 
	$\pi(a,p)$ such that $V[G]$ is a generic extension 
	of $V[G \restrict \kappa_N][K]$ by a proper forcing poset.
	
	Furthermore, letting $J := \pi(\tau_\beta)[K]$, 
	$K^+ := \pi^{-1}(K)$, and $J^+ := \pi^{-1}(J)$, 
	then $K^+$ is a filter on 
	$N[G] \cap (\p_\beta^c \otimes \p_\beta^*)$ containing 
	$(a,p)$ which is $N[G]$-generic, $J^+$ is a filter on 
	$N[G] \cap \p_\beta$ which is $N[G]$-generic, and 
	$J^+ = \tau_\beta[K^+]$. 
	Moreover, there exists $s \in \p_\beta$ such that 
	for all $(b,q)$ in $K^+$, $s \le_\beta^* q$.
	\end{lemma}

\begin{proof}
	By extending further if necessary, we may assume without loss of 
	generality that $a \le_\beta^c p \restrict \even$. 
	Let $\pi(\langle \p_i : i \le \beta \rangle) = 
	\langle \p_i^\pi : i \le \pi(\beta) \rangle$ and 
	$\pi(\langle \dot S_\gamma : \gamma \in \beta \cap \odd \rangle) = 
	\langle \dot S_\gamma^\pi : \gamma \in \pi(\beta) \cap \odd \rangle$. 
	Then the second recursion hypothesis means that in 
	$V[G \restrict \kappa_N]$, for all $\gamma \in \pi(\beta) \cap \odd$, 
	$(\p_\gamma^\pi)^c \otimes (\p_\gamma^\pi)^*$ 
	forces that $\dot S_\gamma^\pi$ is nonstationary 
	in $\kappa_N$. 
	By Proposition 2.9, in $V[G \restrict \kappa_N]$ 
	the forcing poset 
	$\pi(\p_\beta^{*,s})$ 
	contains a $\kappa_N$-closed dense subset.
	
	At stage $\kappa_N$ in the preparation forcing iteration $\R_\kappa$ we forced with 
	$\add(\omega,\kappa_N) \times \add(\kappa_N)$. 
	Therefore, $V[G \restrict (\kappa_N+1)] = V[G \restrict \kappa_N][L]$, where $L$ is a 
	$V[G \restrict \kappa_N]$-generic filter on $\add(\omega,\kappa_N) \times \add(\kappa_N)$. 
	By Lemma 2.10, there exists in 
	$V[G \restrict \kappa_N][L]$ a 
	$V[G \restrict \kappa_N]$-generic filter 
	$K$ on $\pi(\p_\beta^c \otimes \p_\beta^*)$ 
	which contains $\pi(a,p)$ such that 
	$V[G \restrict \kappa_N][L]$ is a generic extension of 
	$V[G \restrict \kappa_N][K]$ by an $\omega_1$-c.c.\! forcing poset. 
	Since $V[G]$ is a generic extension of $V[G \restrict \kappa_N][L]$ by a proper forcing, 
	namely, the tail of the iteration $\R_\kappa$ after forcing with $\R_{\kappa_N + 1}$, it follows 
	that $V[G]$ is a generic extension of $V[G \restrict \kappa_N][K]$ by a proper forcing.
	
	Recall that the map $\tau_\beta : \p_\beta^c \otimes \p_\beta^* \to \p_\beta$ 
	defined by $\tau_\beta(b,q) = q + b$ is a surjective projection mapping by Lemma 1.18. 
	Since $\pi$ is an isomorphism and by absoluteness, 
	in $V[G \restrict \kappa_N]$ we have that 
	$\pi(\tau_\beta)$ is a surjective 
	projection mapping from 
	$\pi(\p_\beta^c \otimes \p_\beta^*)$ 
	onto $\pi(\p_\beta)$. 
	Let $J := \pi(\tau_\beta)[K]$. 
	Then $J$ is a $V[G \restrict \kappa_N]$-generic filter on $\pi(\p_\beta)$.
	
	Let $K^+ := \pi^{-1}(K)$ and $J^+ := \pi^{-1}(J)$. 
	Since $\pi(a,p) \in K$, $(a,p) \in K^+$. 
	It is easy to check that $K^+$ and $J^+$ are filters on 
	$N[G] \cap (\p_\beta^c \otimes \p_\beta^*)$ and 
	$N[G] \cap \p_\beta$ respectively, and 
	$J^+ = \tau_\beta[K^+]$. 
	If $D$ is a dense open subset of $\p_\beta^c \otimes \p_\beta^*$ 
	in $N[G]$, then since $\pi$ is an isomorphism and by absoluteness, 
	$\pi(D)$ is a dense open subset of 
	$\pi(\p_\beta^c \otimes \p_\beta^*)$ in $V[G \restrict \kappa_N]$. 
	Since $K$ is $V[G \restrict \kappa_N]$-generic, we can fix 
	$w \in \pi(D) \cap K$. 
	Then $\pi^{-1}(w) \in D \cap K^+$. 
	This shows that $K^+$ is $N[G]$-generic for 
	$\p_\beta^c \otimes \p_\beta^*$. 
	A similar argument shows that $J^+$ is $N[G]$-generic for $\p_\beta$.
	
	By Lemma 1.16, we can write 
	$V[G \restrict \kappa_N][K] = V[G \restrict \kappa_N][K_1 \times K_2]$, where 
	$K_1 \times K_2 := K \cap 
	(\pi(\p_\beta^c \otimes \p_\beta^*) / \pi(a,p))$ 
	is a $V[G \restrict \kappa_N]$-generic 
	filter on $(\pi(\p_\beta)^c / \pi(a)) \times  (\pi(\p_\beta)^* / \pi(p))$. 
	By Proposition 2.9, 
	$\pi(\p_\beta)^*$ contains a $\kappa_N$-closed 
	dense subset. 
	By standard arguments, it follows that there exists 
	in $V[G \restrict \kappa_N][K]$ 
	a $\pi(\le_\beta^*)$-descending sequence 
	$\langle q_i : i < \kappa_N \rangle$ below $\pi(p)$ 
	which is dense in $K_2$. 
	Let $r_i := \pi^{-1}(q_i)$ for all $i < \kappa_N$. 
	Then $\langle r_i : i < \kappa_N \rangle$ is a 
	$\le_\beta^*$-descending sequence of conditions in 
	$N[G] \cap \p_\beta^*$ below $p$ which is dense in 
	$\pi^{-1}(K_2)$.
	
	Now $\kappa_N$ has cofinality $\omega_1$ in $V[G]$, 
	and since both recursion hypotheses hold for all 
	$\gamma < \beta$, we also have that 
	for all $\gamma < \beta$, $\p_\gamma$ preserves 
	$\omega_1$. 
	By Lemma 2.11, there is $s \in \p_\beta$ such that 
	$s \le_\beta^* r_i$ for all $i < \kappa_N$. 
	Then $s \le_\beta^* r$ for all $r \in \pi^{-1}(K_2)$. 
	Consider $(b,q)$ in $K^+$. 
	Since $(a,p) \in K^+$, without loss of generality 
	$(b,q) \le (a,p)$. 
	Then $\pi(b,q) \in K$, so $\pi(q) \in K_2$. 
	Hence, $q \in \pi^{-1}(K_2)$. 
	Therefore, $s \le_\beta^* q$, which completes 
	the proof.
\end{proof}

The next proposition verifies the second recursion hypothesis. 
We will use the standard result that proper forcings preserve the stationarity 
of stationary subsets of $\alpha \cap \cof(\omega)$, for any ordinal 
$\alpha$ with uncountable cofinality. 
This result is true because any set 
$S \subseteq \alpha \cap \cof(\omega)$ is stationary in 
$\alpha$ iff the set $\{ a \in [\alpha]^\omega : \sup(a) \in S \}$ is stationary 
in $[\alpha]^\omega$, and proper forcings preserve the stationarity of subsets of 
$[\alpha]^\omega$.

\begin{proposition}
	Let $\beta < \omega_3$ be odd, and assume that the two recursion hypotheses hold for 
	all $\gamma \le \beta$. 
	Let $N$ be $(\beta+1)$-suitable and $\pi$ be the transitive collapsing map of $N[G]$. 
	Then for all odd $\gamma \in N \cap (\beta+1)$, in the model $V[G \restrict \kappa_N]$, 
	$\pi(\p_\gamma^c \otimes \p_\gamma^*)$ 
	forces that $\pi(\dot S_\gamma)$ is a nonstationary subset of $\kappa_N$. 
\end{proposition}

\begin{proof}
	Since $N$ is $(\beta+1)$-suitable, $\beta \in N$ by elementarity, so $N$ is also $\beta$-suitable. 
	By the second recursion hypothesis holding at $\beta$, the conclusion of the proposition 
	is true for all odd $\gamma \in N \cap \beta$. 
	So it suffices to show that in $V[G \restrict \kappa_N]$, 
	$\pi(\p_\beta^c \otimes \p_\beta^*)$ 
	forces that $\pi(\dot S_\beta)$ is a nonstationary 
	subset of $\kappa_N$.

	Let $(a_0,p_0) \in \pi(\p_\beta^c \otimes \p_\beta^*)$, 
	and we will 
	find $(a,p) \le (a_0,p_0)$ 
	which forces that $\pi(\dot S_\beta)$ is nonstationary in $\kappa_N$. 
	By extending further if necessary, we may assume without loss of 
	generality that $a_0 \le p_0 \restrict \even$ in $\pi(\p_\beta)^c$. 
	Then by Lemma 1.15, 
	$\pi(\p_\beta^c \otimes \p_\beta^*) / (a_0,p_0)$ 
	is equal to 
	the product forcing $(\pi(\p_\beta)^c / a_0) \times 
	(\pi(\p_\beta)^* / p_0)$.

	Let $K$, $J$, $K^+$, $J^+$, and $s$ be as described 
	in Lemma 3.7, where $(a_0,p_0) \in K$. 
	Use $J^+$ to interpret the name $\dot S_\beta$ by 
	letting $S$ be the set of $\alpha < \kappa_N$ 
	such that for some $u \in J^+$, 
	$u \Vdash_\beta \check \alpha \in \dot S_\beta$. 
	We claim that $S = \pi(\dot S_\beta)^J$. 
	Clearly $\pi(\dot S_\beta)^J$ is a subset of $\kappa_N$, 
	since $\pi(\dot S_\beta)$ is a 
	$\pi(\p_\beta)$-name for a subset of $\kappa_N$.

	Consider $\alpha  < \kappa_N$. 
	In $V[G]$, let $D$ be the dense open set of conditions 
	in $\p_\beta$ which decide whether or not 
	$\alpha$ is in $\dot S_\beta$. 
	By the elementarity of $N[G]$, $D \in N[G]$. 
	Since $J^+$ is $N[G]$-generic, fix $w \in J^+ \cap D$. 
	Let $w' := \pi(w)$, which is in $\pi(D)$. 
	Since $\pi$ is an isomorphism and by absoluteness, 
	$w'$ decides in $\pi(\p_\beta)$ 
	whether or not $\pi(\alpha) = \alpha$ 
	is in $\pi(\dot S_\beta)$ the same way that $w$ 
	decides whether $\alpha$ is in $\dot S_\beta$.
	As $J$ and $J^+$ are filters, it easily follows that 
	$\alpha \in S$ iff 
	$w \Vdash_\beta^{V[G]} \alpha \in 
	\dot S_\beta$ iff 
	$w' \Vdash_{\pi(\p_\beta)}^{V[G \restrict \kappa_N]} 
	\alpha \in \pi(\dot S_\beta)$ iff 
	$\alpha \in \pi(\dot S_\beta)^J$. 
	Thus, $S = \pi(\dot S_\beta)^J$.
	
	By the choice of $\dot S_\beta$, we know that 
	in $V[G]$ the forcing poset $\p_\beta$ forces that 
	$\dot S_\beta$ does not reflect to any ordinal 
	in $\kappa$ with cofinality $\omega_1$. 
	Now $\kappa_N$ has cofinality $\omega_1$ in $V[G]$, 
	and by the recursion hypotheses $\omega_1$ is 
	preserved by $\p_\beta$. 
	Thus, $\p_\beta$ forces that there exists a club subset 
	of $\kappa_N$ with order type $\omega_1$ which is 
	disjoint from $\dot S_\beta \cap \kappa_N$. 
	Let $\dot c$ be a $\p_\beta$-name for such a club.

	By the first recursion hypothesis holding 
	for $\beta$, 
	$\p_\beta^*$ is $\omega_2$-distributive in $V[G]$. 
	By Proposition 2.4, we can find $t \le_\beta^* s$ 
	and a $\p_\beta^c$-name $\dot c_0$ such that 
	$t \Vdash_\beta \dot c = \dot c_0$. 
	By the maximality principle for names, 
	we may assume without 
	loss of generality that $\dot c_0$ is a 
	$\p_\beta^c$-name for a club 
	subset of $\kappa_N$ with order type $\omega_1$. 
	As $\p_\beta^c$ is $\omega_1$-c.c., we can find a 
	set $d$ in $V[G]$ which is a club subset of 
	$\kappa_N$ such that $\p_\beta^c$ forces that 
	$d \subseteq \dot c_0$. 
	Then 
	$t \Vdash_\beta d \cap \dot S_\beta = \emptyset$.
	
	We claim that $d \cap S = \emptyset$. 
	If not, then fix $\alpha \in d \cap S$. 
	By the definition of $S$, there exists 
	$u \in J^+$ which forces in $\p_\beta$ that 
	$\alpha$ is in $\dot S_\beta$. 
	Since $J^+ = \tau_\beta[K^+]$ by Lemma 3.7, 
	there is $(b,z) \in K^+$ such that $u = z + b$.
	
	By Lemma 3.7, $s \le_\beta^* z$. 
	So $t \le_\beta^* z$. 
	By Lemma 1.10(3), $t$ and $b$ are compatible in 
	$\p_\beta$ and 
	$t + b \le_\beta z + b = u$. 
	It follows that $t + b$ forces in $\p_\beta$ that 
	$\alpha \in \dot S_\beta$. 
	This is impossible since $\alpha \in d$ and 
	$t$ forces in $\p_\beta$ 
	that $d \cap \dot S_\beta = \emptyset$.

	So indeed $d \cap S = \emptyset$, and hence $S$ is a 
	nonstationary subset of $\kappa_N$ in the model $V[G]$. 
	Since $S = \pi(\dot S_\beta)^J$, 
	$S \in V[G \restrict \kappa_N][J]$. 
	As $V[G \restrict \kappa_N][J] \subseteq 
	V[G \restrict \kappa_N][K]$, 
	$S \in V[G \restrict \kappa_N][K]$. 
	But $V[G]$ is a generic extension 
	of $V[G \restrict \kappa_N][K]$ by 
	a proper forcing poset by Lemma 3.7. 
	Since $S$ is a set of ordinals of cofinality $\omega$, 
	$S$ must be nonstationary in 
	$V[G \restrict \kappa_N][K]$. 
	As $(a_0,p_0) \in K$, we can find 
	$(a,p) \le (a_0,p_0)$ in $K$ which forces in 
	$\pi(\p_\beta^c \otimes \p_\beta^*)$ that 
	$\pi(\dot S_\beta)$ is nonstationary in $\kappa_N$, 
	which completes the proof.
\end{proof}

We now verify the first recursion hypothesis for $\beta$, 
which will finish the proof of the 
consistency result.

\begin{proposition}
	Let $\beta < \kappa^+$, and assume that the 
	first and second recursion hypotheses hold for 
	all $\gamma < \beta$  
	and the second recursion hypothesis 
	holds for $\beta$. 
	Then $\p_\beta^*$ is $\omega_2$-distributive.
	\end{proposition}

\begin{proof}
	Assume that $p \in \p_\beta$ forces 
	in $\p_\beta^*$ that 
	$\langle \dot \alpha_i : i < \omega_1 \rangle$ 
	is a sequence of ordinals. 
	We will find $q \le_\beta^* p$ which decides 
	in $\p_\beta^*$ the value of 
	$\dot \alpha_i$ for all $i < \omega_1$, and hence 
	forces that this sequence is in the ground model.
	
	Fix a $\beta$-suitable model $N$ such that $N[G]$ 
	contains $p$ and $\langle \dot \alpha_i : i < \omega_1 \rangle$, 
	and let $\pi$ be the transitive collapsing map of $N[G]$. 
	Fix $K$, $J$, $K^+$, $J^+$, and $s$ as in Lemma 3.7, 
	where $\pi(p \restrict \even,p) \in K$. 
	Then $(p \restrict \even,p) \in K^+$.
	
	Let $i < \omega_1$, and we will show that 
	$s$ decides the value of $\dot \alpha_i$. 
	Let $D$ be the set of 
	$(b,q) \in \p_\beta^c \otimes \p_\beta^*$ below 
	$(p \restrict \even,p)$ such that 
	$q$ decides in $\p_\beta^*$ the value of 
	$\dot \alpha_i$. 
	Then $D \in N[G]$ by elementarity, and easily 
	$D$ is dense below $(p \restrict \even,p)$. 
	Since $K^+$ is $N[G]$-generic and contains 
	$(p \restrict \even,p)$, fix $(b,q) \in D \cap K^+$. 
	Then by Lemma 3.7, $s \le_\beta^* q$. 
	Since $q \in D$, 
	$q$ decides the value of $\dot \alpha_i$. 
	So $s$ decides the value of $\dot \alpha_i$.
	\end{proof}

\smallskip

\begin{center}\textsc{Postscript}\end{center}

\smallskip

After this article was completed, I.\ Neeman discovered a shorter proof of the consistency of stationary set 
reflection at $\omega_2$ together with an 
arbitrarily large continuum. 
Specifically, starting with a model in which stationary 
set reflection holds, adding any number of Cohen reals 
preserves stationary set reflection 
(see \cite[Theorem 3.1]{jk35}). 
This new proof is, however, somewhat 
limited in its applications. 
For example, in the model of Harrington and Shelah \cite{HS}, 
there exists a special $\omega_2$-Aronszajn tree, and that 
fact cannot be changed by Cohen forcing. 
In contrast, the methods developed in this paper can be 
used to construct models of stationary set reflection 
together with a variety of combinatorial properties, 
such as the non-existence of special 
$\omega_2$-Aronszajn trees and the failure of 
the weak Kurepa hypothesis. The details will be 
worked out in the first author's upcoming 
Ph.D.\ dissertation.

\bibliographystyle{plain}
\bibliography{paper32}

\begin{thebibliography}{1}

\bibitem{abraham}
U.~Abraham.
\newblock Aronszajn trees on $\aleph_{2}$ and $\aleph_{3}$.
\newblock {\em Ann. Pure Appl. Logic}, (3):213--230, 1983.

\bibitem{jk35}
T.~Gilton and J.~Krueger.
\newblock A note on the eightfold way.
\newblock Preprint.

\bibitem{HS}
L.~Harrington and S.~Shelah.
\newblock Some exact equiconsistency results in set theory.
\newblock {\em Notre Dame J. Formal Logic}, 26(2):178--188, 1985.

\bibitem{mitchell}
W.~Mitchell.
\newblock Aronszajn trees and the independence of the transfer property.
\newblock {\em Ann. Math. Logic}, 5:21--46, 1972/73.

\bibitem{boban}
Veli\u{c}kovi\'{c}.
\newblock Forcing axioms and cardinal arithmetic.
\newblock In {\em Logic Colloquium 2006}, volume~32 of {\em Lect. Notes Log.},
  pages 328--360. Assoc. Symbol. Logic, Chicago, IL, 2009.

\end{thebibliography}

\end{document}